\documentclass[11pt,a4paper]{amsart}
\usepackage[utf8]{inputenc}
\usepackage[english]{babel}
\usepackage{amsmath}
\usepackage{amsfonts}
\usepackage{amssymb}
\usepackage{amsthm}
\usepackage{tikz-cd}
\usepackage{float}

\usepackage[colorlinks=true]{hyperref}
\usepackage{latexsym}    
 \usepackage{amssymb}   
\usepackage{amsmath}    
\usepackage{amsbsy}
\usepackage{amsthm}
\usepackage{amsgen}
\usepackage{amsfonts}
\usepackage{array}
\usepackage{lmodern}
\usepackage[all]{xy}    

\usepackage{caption}
\usepackage[labelformat=simple]{subcaption}

\newsavebox{\largestimage}

\usepackage{tikz}
\usetikzlibrary{knots}
\usetikzlibrary{decorations.pathmorphing}
\usetikzlibrary{calc, decorations.markings}

\usepackage{color}
\usepackage{verbatim}
\usepackage{url}
\usepackage{enumerate}
\numberwithin{figure}{section}

\usepackage{lineno}

\tikzset{snake it/.style={decorate, decoration=snake}}

\DeclareMathOperator{\Diff}{Diff}

\DeclareMathOperator{\codim}{codim}

\DeclareMathOperator{\kernel}{Ker}

\DeclareMathOperator{\dime}{dim}
\DeclareMathOperator{\proj}{proj}
\DeclareMathOperator{\reg}{reg}
\DeclareMathOperator{\prin}{prin}

\newtheorem{theorem}{Theorem}  
\newtheorem{corollary}[theorem]{Corollary}  
		
		\newtheorem{thm}{Theorem}[section]
		
		\newtheorem{lem}[thm]{Lemma}
		\newtheorem{cor}[thm]{Corollary}
		
		\newtheorem{prop}[thm]{Proposition}
	\theoremstyle{definition}	
		\newtheorem{remark}[thm]{Remark}

 \newtheoremstyle{TheoremNum}
        {\topsep}{\topsep}              
        {\itshape}                      
        {}                              
        {\bfseries}                     
        {.}                             
        { }                             
        {\thmname{#1}\thmnote{ \bfseries #3}}
    \theoremstyle{TheoremNum}
    \newtheorem{duplicate}{Theorem}

\numberwithin{equation}{section}

\theoremstyle{definition}
\newtheorem*{ack}{Acknowledgements}
\newtheorem*{question}{Question}
		
\usepackage{tikz}

\usepackage{hyperref}
\hypersetup{
    colorlinks=true,
    linkcolor=blue,
    citecolor=blue,
    urlcolor=blue,
    pdfauthor={Diego Corro},
    pdftitle={CLASSIFICATION OF A-FOLIATIONS OF CODIMENSION 2 ON SIMPLY-CONNECTED MANIFOLDS}
}

\newcommand{\R}{\mathbb{R}} 
\newcommand{\Z}{\mathbb{Z}} 
\newcommand{\C}{\mathbb{C}} 

\newcommand{\T}{\mathbb{T}}

\newcommand{\Sp}{\mathbb{S}} 
\newcommand{\D}{\mathbb{D}} 

\newcommand{\fol}{\mathcal{F}} 

\begin{document}
		
\author[D.~Corro]{Diego Corro$^{*}$}
\address[D.~CORRO]{Institut f\"ur Algebra und Geometrie, Karlsruher Institut f\"ur Technologie (KIT), Karlsruhe, Germany.}
\email{\href{mailto:diego.corro.math@gmail.com}
{diego.corro.math@gmail.com}}
\urladdr{\url{http://www.diegocorro.com}}
\thanks{$^*$Supported by CONACyT-DAAD Scholarship number 409912, DFG ($281869850$, RTG $2229$--“Asymptotic Invariants and Limits of Groups and Spaces”), DGAPA-Fellowship associated to the Mathematics Institute of
UNAM, campus Oaxaca, and DFG-Eigenestelle Fellowship CO 2359/1-1.}


\title[A-FOLIATIONS OF CODIMENSION 2]{A-FOLIATIONS OF CODIMENSION TWO ON COMPACT SIMPLY-CONNECTED MANIFOLDS}
\date{\today}

\subjclass[2010]{53C12, 57R30, 53C24}
\keywords{singular Riemannian foliation, aspherical foliations, torus actions}

\setlength{\overfullrule}{5pt}
	\begin{abstract}
		We show that a singular Riemannian foliation of codimension two on a compact simply-connected Riemannian $(n+2)$-manifold, with regular leaves homeomorphic to the $n$-torus, is given by a smooth effective $n$-torus action. This solves in the negative for the codimension $2$ case a question about the existence of foliations by exotic tori on simply-connected manifolds.
	\end{abstract}
	
\maketitle

\section{Main results}

When studying a Riemannian manifold $M$, an approach to understand	 its geometry or its topology  is to simplify the problem by ``reducing" $M$ to a lower dimensional space $B$. This can be achieved by considering a partition of the original manifold $M$ into submanifolds which are, roughly speaking, compatible with the Riemannian structure of $M$. This ``reduction" approach is encompassed in the concept of \emph{singular Riemannian foliations}.

This reduction approach has been applied to the long-standing open problem in Riemannian geometry of classifying and constructing  Riemannian manifolds of positive or nonnegative (sectional) curvature via the \emph{Grove symmetry program}, when the foliation is given by  an \emph{effective isometric action by a  compact Lie group}. When  the leaves of a singular Riemannian foliation are given by the orbits of a smooth Lie group action we say that the foliation is a \emph{homogeneous foliation}. By \cite{Dirk1981} and \cite{Radeschi2014} it is clear that the concept of a singular Riemannian foliation is more general than the one of a Lie group action.

Since any compact connected Lie group contains a maximal torus as a Lie subgroup, the study of torus actions is of importance in the study of homogeneous foliations. The classification up to equivariant diffeomorphism of smooth, closed, simply-connected, manifolds with torus actions is a well studied problem when either the dimension of the manifolds or the cohomogeneity of the action is low  (see for example \cite{Orlik1970}, \cite{Kim1974},  \cite{Fintushel77}, \cite{Oh1983}, \cite{Oh1982}).

One main difference between smooth group actions and foliations is that  foliations may be less rigid, not having several constraints natural to Lie groups (see for example \cite{GeRadeschi2013}). This in turn raises technical challenges, such as the fact that the leaves may carry non-standard smooth structures. Thus an important problem in the setting of singular Riemannian foliations is to distinguish homogeneous foliations from non-homogeneous ones (see for example \cite{GeRadeschi2013}).  This problem does not become more tractable when the topology (and geometry) of the manifold is simple. Even in the case of  spheres, equipped with the round metric, it is not clear how to distinguish homogeneous foliations (i.e. those coming from group actions) from non-homogeneous ones (see for example \cite{Siffert2017}).

By focusing on compact, simply-connected manifolds with a singular Riemannian foliations with closed aspherical leaves, we are able to attack this general problem. This type of singular Riemannian foliations are denoted as \emph{$A$-foliations} and they were introduced by Galaz-Garc\'{i}a and Radeschi in \cite{Galaz-Garcia2015} as generalizations of smooth effective torus actions on smooth manifolds. 

The main result of the present work is that $A$-foliations of codimension two on compact, simply-connected manifolds are homogeneous up to foliated diffeomorphism.

\begin{theorem}\label{T: Codimn 2 A foliation is diffeo to homogeneous fol}
Every $A$-foliation of codimension two on  a compact, simply-connected, Riemannian $n$-manifold, with $n\geqslant 3$,  is homogeneous.
\end{theorem}

In the codimension one case, the same result holds. Namely, Galaz-Garc\'{i}a and Radeschi in \cite{Galaz-Garcia2015} give a classification up to foliated diffeomorphism of all compact, simply-connected manifolds with a codimension one $A$-foliations. They show that these foliations are homogeneous. 

They also point out that the lack of examples of $A$-foliations with exotic tori as leaves on simply-connected manifolds, and ask the following question:
\begin{question}
Does there exist a non-trivial example of a singular Riemannian foliation whose leaves are exotic tori?
\end{question}	
In this context, Theorem~\ref{T: Codimn 2 A foliation is diffeo to homogeneous fol} gives an answer in the negative to this question in the case that the foliation has codimension two. 
	
A possible approach to solve this question in the positive follows from the results in \cite{Galaz-Garcia2015}: we may consider a fiber bundle with fiber an exotic torus, over a manifold with sufficiently large second homotopy group. This approach was pursued by Farrell and Wu in \cite{FarrellWu2018} considering as base spaces $4$-dimensional simply-connected manifolds with large second homotopy groups, but they only managed to produce manifolds with finite fundamental group. Moreover,  the authors comment that they do not know of any smooth fiber bundle  such that the fiber is an exotic torus and the total space and base space are smooth simply-connected manifolds. So at the present there is no clear way to decide this question in general.

To prove Theorem~\ref{T: Codimn 2 A foliation is diffeo to homogeneous fol} we extend results of the theory of transformation groups to the setting of singular Riemannian foliations. We  focus on the general problem of comparing two different manifolds, each one endowed with a singular Riemannian foliation, via the leaf space, which is the topological space obtained as a quotient of the foliated manifold by the equivalence relation given by the foliation. A technique for classifying  compact manifolds admitting a smooth effective compact Lie group action,  up to homeomorphism is to compare their orbit spaces (see for example \cite{Orlik1970}, \cite{Kim1974},\cite{Fintushel77}, \cite{Oh1983},\cite{Grove2012}). We apply the same idea to smooth manifolds admitting a singular Riemannian foliation.

We begin by studying the homeomorphism type of the leaves of $A$-folia\-tions on compact, simply-connected manifolds. We prove in Proposition~\ref{P: Holonomy implies Bieberbach type} that, except in the case of $4$-dimensional leaves,  the leaves are all homeomorphic to tori or Bieberbach manifolds, extending results in \cite{Galaz-Garcia2015}.

\begin{theorem}
Let $(M,\fol)$ be an $A$-foliation, and consider a leaf $L\in \fol$. If $\dim(L)\neq 4$, then $L$ is homeomorphic to a Biberbach manifold. 
\end{theorem}

This result is due to the positive answer to the Borel conjecture for virtually abelian groups, except for dimension $4$ (see for example \cite{FarrellHsiang1983}). In \cite{Galaz-Garcia2018} the authors introduce $B$-foliations. These are $A$-foliations such that the leaves are Biberbach manifolds. In light of Proposition~\ref{P: Holonomy implies Bieberbach type}, from a topological view point any $A$-foliation without $4$-dimensional leaves is a $B$-foliation. 
Since Biberbach manifolds are charaterized as those closed manifolds admitting a flat Riemannian metric \cite[Corollary 5.1]{Charlap}, we propose here to modify the definition of $B$-foliation in \cite{Galaz-Garcia2015} to the following one: A $B$-foliation is an $A$-foliation such that the leaves with the induced Riemannian metric are flat, see Remark \ref{R: New definition of B-foliations}.

We also study the infinitesimal foliations of an $A$-foliation, as well as the holonomy of the leaves. This allows us to describe the tubular neighborhoods of the singular leaves, and we propose a finer stratification of the manifold. Both concepts of holonomy and the infinitesimal foliation in the case of homogeneous foliations are encoded in the isotropy subgroup of an orbit. For an $A$-foliation of codimension $2$ on a simply-connected manifold we define \emph{the weights of the foliation}, which encode the information of the infinitesimal foliation. 
The weights defined in this present work generalize the weights of smooth effective torus actions (defined   in \cite{Orlik1970}, \cite{Fintushel77}, \cite{Oh1983}), which encode the isotropy information of torus actions.

We say that two weighted leaf spaces are \emph{isomorphic}, if there is a weight preserving homeomorphism, and prove that for compact simply-connected manifolds with singular $A$-foliations of codimension $2$ the weighted orbit spaces determine up to foliated homeomorphism the foliated manifold. 

\begin{theorem}\label{T: weights classify the foliation}
Let $(M_1 , \fol_1 )$ and $(M_2 , \fol_2 )$ be compact , simply-connected Rie\-mannian manifolds with singular $A$-folia\-tions of codimension $2$.  If the leaf spaces $M_1^\ast$ and $M_2^\ast$ are isomorphic, then $(M_1,\fol_1)$ is foliated homeomorphic to $(M_2,\fol_2)$.
\end{theorem}

We point out that in the general  setting of classifying manifolds with singular Riemannian foliations via their leaf spaces, the best one can obtain is a classification up to foliated homeomorphism. This is because the leaf spaces are in general only metric spaces (i.e. they may not even be topological manifolds). 
  
For the proof of Theorem~\ref{T: weights classify the foliation} we need the existence of a \emph{cross-section}. This is a map $\sigma \colon M^\ast\to M$ such that for the natural projection $\pi\colon M\to M^\ast$ we have $\pi\circ \sigma = \mathrm{Id}_{M^\ast}$. Theorem \ref{T: first obstructions} we give a family of topological obstruction to the existence of a cross-section over a subset of the principal stratum $M_{\prin}^\ast$ for a general closed singular Riemannian foliation, and in Theorem \ref{T: second obstruction} we give a family of obstructions to extend a cross-section on a closed subset of the principal stratum to the whole leaf space $M^\ast$. Combining these obstructions, we give a sufficient condition for the existence of a cross-section in the following  corollary (see Section~\ref{S: obstruction to cross-sections} for the necessary definitions):

\begin{corollary}\label{C: existence of cross sections intro}
Let $(M, \fol)$ be a closed singular Riemannian foliation on a simply-connected manifold. Suppose that we have a cross-section $\tilde{\sigma}\colon  M_{\prin}^\ast\to M$, that the homotopy fiber $F_\pi$ of $\pi\colon M\to M^\ast$ is simple, and that there exist $A^\ast\subset M_{\prin}^\ast$ closed, such that $(M^\ast,A^\ast)$ is a $CW$-pair. If $M_{\prin}^\ast$ has the same homotopy type as $M^\ast$, then the cross-section $\tilde{\sigma}$ can be extended to a section $\sigma\colon M^\ast\to M$.
\end{corollary}

For the case of a smooth effective torus action of cohomogeneity two (i.e. a homogeneous foliation), on a compact simply-connected manifold Oh  proved in \cite{Oh1983} that  the leaf space is a weighted $2$-disk, with the weights satisfying some conditions. Furthermore, he proves that these conditions characterize  the orbit spaces of such actions among all weighted $2$-disks. Namely, given a weighted disk satisfying the conditions he gives  a procedure to construct a closed simply-connected smooth manifold with an effective smooth torus action of cohomogeneity two realizing the weighted disk as an orbit space. Oh called such weighted $2$-disks \emph{legally weighted}.

We  show that the weights of an $A$-foliation of codimension two on a compact, simply-connected manifold $M$ are legal weights in the sense of Oh. Moreover for a singular $A$-foliation of codimension two on a compact, simply-con\-nected manifold   it was proved in \cite{Galaz-Garcia2015} that the leaf space  is homeomorphic to a $2$-dimensional disk, and the boundary points of the leaf space correspond exactly to the singular leaves of the foliation. By studying closely the weight of the foliation we conclude that the results in \cite{Oh1983} apply. Namely, that there is a torus action on $M$ with the same weights as the ones of the  foliation. By Theorem~\ref{T: weights classify the foliation} we conclude that a singular $A$-foliation of codimension two on a compact, simply-connected manifold is, up to foliated homeomorphism, a homogeneous foliation.

As mentioned before,  in the problem of classifying manifolds with singular Riemannian foliations via  their leaf spaces, in general the best one can obtain is a classification up to foliated homeomorphism. In the case of singular $A$-foliations of codimension two  on compact, simply-connected manifolds since the leaf space is a $2$-disk, it is a smooth manifold with boundary in a unique way (it admits a unique smooth structure). Thus we can expect in this case to get a classification up to foliated diffeomorphism.

The next obstacle to obtaining a smooth classification, and obtaining Theorem~\ref{T: Codimn 2 A foliation is diffeo to homogeneous fol}, is the existence of exotic smooth structures on tori (see, for example \cite{Hsiang1969}, \cite{Hsiang}). As stated before, there exists regular $A$-foliations of codimension $4$ on compact manifolds with finite fundamental group  and leaves consisting of exotic tori \cite{FarrellWu2018}.  
To finish the proof of Theorem~\ref{T: Codimn 2 A foliation is diffeo to homogeneous fol} we study the diffeomorphism type of the leaves of a singular $A$-foliation of codimension two on a compact simply-connected smooth manifold, and prove that they are diffeomorphic to standard tori. 

We remark that $A$-foliaions of codimension $2$ with singular leaves are a priori infinitesimally polar due to a general argument by Lytchak \cite[Proposition 3.1]{Lytchak2010}. But since they are given by torus actions  by Theorem~\ref{T: Codimn 2 A foliation is diffeo to homogeneous fol}, they are polar by \cite[Example 4.4]{Grove2012}. Grove and Ziller showed in \cite{Grove2012} that a \emph{Coxeter polar action} is determined by the orbit space together with the isotropy information. This is an analogous statemtent of Theorem~\ref{T: weights classify the foliation} for such polar actions (in this case the weights are the isotropy information of each orbit). Thus it is natural to if an analogous statement to Theorem~\ref{T:Family of Obstructions} holds for more general $A$-foliations.
\begin{question}
Are polar closed singular Riemannian foliations (or $A$-foliations) determined by the leaf space, and information about the infinitesimal foliation over the strata corresponding to the singular leaves?
\end{question}

Our article is organized as follows. In Section~\ref{S: Preliminaries} we give an overview of the theory of Lie group actions and singular Riemannian foliations, such as the infinitesimal foliation and the holonomy. In Section~\ref{S: A-foliations} we study the homeomorphism type of the leaves of a general $A$-foliation on a compact simply-connected manifold. In Section~\ref{S: A-foliations of codimension 2}  we define the weights of an $A$-foliation and prove  Theorem~\ref{T: weights classify the foliation}. We finish the proof of Theorem~\ref{T: Codimn 2 A foliation is diffeo to homogeneous fol} in Section~\ref{S: smooth structure of leaves}, where we study the diffeomorphism type of the leaves of a singular $A$-foliation of codimension two on a simply-connected manifold. We end this manuscript with the  presentation of the topological sufficient conditions for the existence of cross-sections in  Section~\ref{S: obstruction to cross-sections}.\\[1em]


\begin{ack} I would like to thank Caterina Campagnolo, Fernando Galaz-Garc\'{i}a, Karsten Grove, Jan-Bernard Korda\ss, Adam Moreno  and Wilderich Tuschmann for helpful conversations on the  constructions presented herein. I thank Alexander Lytchak and Oliver Görtsches for helpful discussions which led to the obstruction results for the existence of cross-sections. I thank Thomas Farrell and Marco Radeschi for discussions and comments about the smooth structure of the leaves. Last I thank the anonymous referees for giving suggestions to improve the presentation of the manuscript, and the suggestion for the study of polar $A$-foliations.
\end{ack}	

\section{Preliminaries}\label{S: Preliminaries}


\subsection{Group actions}
Let $G\times M \to M$, $p \mapsto g \star p$, be a smooth action of a compact Lie group $G$ on a smooth manifold $M$. The \emph{isotropy group}  at $p$ is defined 	as $G_p = \{g\in G \mid g \star p = p \}$. We say that the orbit $G(p)$ is \emph{principal} if the isotropy group $G_p$ acts trivially on the normal space to the orbit at $p$. It is a well known fact that the set of principal orbits  is open and dense in $M$. Since the isotropy groups	of principal orbits are conjugate in $G$, and since the orbit $G(p)$ is diffeomorphic to $G/G_p$, all principal orbits have the same dimension. If $G(p)$ has the same dimension as a principal orbit but the isotropy group acts non-trivially on the normal space to the orbit at $p$ we say that the orbit is \emph{exceptional}. If the dimension of the orbit $G(p)$ is less than the dimension of a principal orbit, we say that the orbit is \emph{singular}. We denote the set of exceptional orbits by $E$ and the set of singular orbits by $Q$. We denote the \emph{orbit space} $M/G$ by $M^\ast$ and we define the \emph{cohomogeneity} of the action to be the dimension of the orbit space $M^\ast$ (or, equivalently, the codimension in $M$ of a principal orbit). Let $\pi \colon M \to M^\ast$ be the orbit projection map onto the orbit space. We denote by  $X^\ast$ the image of a subset $X$ of $M$ under the orbit projection map $\pi$. The action is called \emph{effective} if the intersection of all isotropy subgroups of the action is trivial, i.e if $\cap_{p\in M} G_p = \{e\}$. We say that a Riemannian metric is \emph{invariant} under the action if the group acts by isometries with respect to this metric. For every effective smooth action of a compact Lie group $G$ on a smooth manifold $M$ there exists an invariant Riemannian metric (see, for example  \cite[Theorem~3.65]{Alexandrino}).

\subsection{Effective torus actions}\label{S: torus actions}
	
	
From now on we denote the $n$-torus with a Lie group structure by $\T^n$ in order to distinguish it from its underlying topological space $T^n$. By identifying the torus group $\T^n$ with $\R^n / \Z^n$ we note that a circle subgroup of $\T^n$ is determined by a line through the origin given by a vector $\overline{v} = (a_1, \ldots , a_n )$ $\in \Z^n$, with $a_1, \ldots, a_n$ relatively prime, via $G(a) = \{(e^{2\pi i t a_1}, \ldots , e^{2\pi i t a_n} ) \mid 0 \leqslant t \leqslant 1 \}$ (for a more detailed discussion see \cite{Oh1983}).  Recall that a smooth, effective action of a torus on a smooth manifold always has trivial principal isotropy. Therefore, a smooth, effective action of an $n$-torus on a smooth $(n+2)$-manifold  has cohomogeneity two.
		
Let $M$ be a closed, simply-connected, smooth  $(n+2)$-manifold, $n\geqslant 2$, on which a compact Lie group $G$ acts smoothly and effectively with cohomogeneity two. It is well known  that, if the set $Q$ of singular orbits is not empty, then the orbit space $M^\ast$ is homeomorphic to a $2$-disk $D^\ast$ whose boundary is $Q^\ast$ (see \cite[Chapter~IV]{Bredon}). Moreover, the 	interior points correspond to principal orbits (i.e.~the action has no exceptional orbits). The orbit space structure was analyzed in \cite{Kim1974}, \cite{Oh1983} when 	$G = \T^n$ for $n \geqslant 2$ (see also \cite{Galaz-GarciaKerin2014}). In this case the only possible non-trivial 	isotropy groups are circle subgroups and $2$-torus subgroups of $\T^n$. Furthermore, the boundary  $Q^\ast$  decomposes as a a finite union of $m \geqslant n$ open edges $\Gamma_i$ and $m$-vertices $F_i$, where $F_i$ sits between $\Gamma_{i-1}$ and $\Gamma_i$ if $i\geqslant 2$ and $F_1$ sits between $\Gamma_{m}$ and $\Gamma_1$. The interior points in an edge $\Gamma_i$ correspond to  orbits with a fixed circle subgroup $G(\overline{v}_i)$ determined by a vector $\overline{v}_i  = (a_{1i},\ldots,a_{ni})\in \Z^n$. For $i\geqslant 2$, the orbit corresponding to a vertex $F_i$ is an orbit with isotropy equal to the product of circle 	subgroups $G(\overline{v}_i) \times G(\overline{v}_{i+1})$, i.e. a $2$-torus in $\T^n$; for $i = 1$, the isotropy of the orbit corresponding to $F_1$ is given by $G(\overline{v}_m)\times G(\overline{v}_1)$.	This structure is illustrated in Figure~\ref{F:ORBIT_SPACE} (cf.~\cite[Figure~1]{Galaz-GarciaKerin2014}).
		
	
		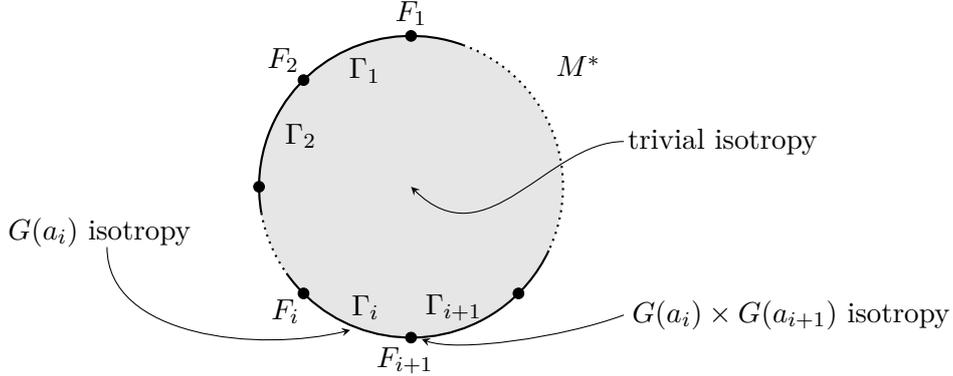
\begin{figure}
		  \begin{tikzpicture}[scale=2]
		    \path[coordinate] (0.5,0.866)  coordinate(A) 
		                (90:1) coordinate(B)
		                (-0.5,0.866) coordinate(C)
		                (135:1) coordinate(D)
		                (-1,0) coordinate(E)
		                (-0.866,-0.5) coordinate(F)
		                (225:1) coordinate(G)
		                (0,-1) coordinate(H)
		                (315:1) coordinate(I);

		      \fill[gray!20]	([shift=(30:1cm)]0,0) arc (30:390:1cm);
		    	\draw[thick] ([shift=(70:1cm)]0,0) arc (70:190:1cm);
		    	\draw[dotted,thick] ([shift=(190:1cm)]0,0) arc (190:215:1cm);	

		    	\draw[thick] ([shift=(215:1cm)]0,0) arc (215:335:1cm);
		    \draw[dotted,thick] ([shift=(70:1cm)]0,0) arc (70:-25:1cm);
		
		    \draw[color = black] (1.1, 0.8) node{$M^\ast$};
		    
		    \draw[color = black] (111.5:0.83) node{$\Gamma_1$};
		    \draw[color = black] (155:0.8) node{$\Gamma_2$};
		    \draw[color = black] (-0.3,-0.8) node{$\Gamma_i$};
		    \draw[color = black] (0.28,-0.8) node{$\Gamma_{i+1}$};
		    \draw[color = black] (90:1.15) node{$F_1$}; 
		    \draw[color = black] (135:1.19) node{$F_2$};
		    \draw[color = black] (225:1.17) node{$F_i$};
		    \draw[color = black] (268:1.15) node{$F_{i+1}$};
		
		    \draw[color = black] (2.05, 0.3) node{trivial isotropy};
		    \draw[stealth-] (0,0) .. controls +(0.5,-0.5) and +(-0.5,0) .. (1.4,0.3);
		
		    \draw[color = black] (2.5,-0.85) node{$G(a_{i}) \times G(a_{i+1})$ isotropy};
		    \draw[-stealth] (1.4,-0.85) .. controls (1,-1) and (0.3,-1.1) .. (0.07,-1.01);
		
		    \draw[color = black] (-2.05, -0.3) node{$G(a_i)$ isotropy};
		    \draw[-stealth] (-2,-0.4) .. controls (-2,-1) and (-1,-1.1) .. (-0.4,-0.95);
		
		    \filldraw[black] 
		                     (B) circle (1pt)
		                     (D) circle (1pt)
		                     (E) circle (1pt)
		                     (G) circle (1pt)
		                     (H) circle (1pt)
		                     (I) circle (1pt);
		
		  \end{tikzpicture}
		  		\caption{Orbit space structure of a cohomogeneity-two torus action on a 
		  		closed, simply-connected manifold.}
				\label{F:ORBIT_SPACE}
		  \end{figure}	  

The vectors $\{\overline{v}_1,\ldots, \overline{v}_m \}$  of the isotropy invariants $G(\overline{v}_i)$ are called the \emph{weights} of the orbit space. The weight $\overline{v}_i$ associated to a singular orbit is given by the following principal bundle (see \cite[Proposition~3.41]{Alexandrino}):
	\begin{linenomath}
	\begin{align}
		G(\overline{v}_i)\to \T^n \to \T^{n-1}.
	\end{align}
	\end{linenomath}
The weight $\overline{v}_i$ determines the embedding of the isotropy subgroup $G(\overline{v}_i)$ into $\T^n$.	Following Oh \cite{Oh1982}, we say that the the orbit space  $M^*$ is \emph{legally weighted}, if we can find a sub-collection of $n$ weights $\{\overline{v}_{i_1},\overline{v}_{i_2},\ldots, \overline{v}_{i_n} \big\}$ such that the matrix  
	 \[
\begin{pmatrix}
	 \mid & \mid   & \mid  & \mid  \\
	\bar{v}_{i_1} & \bar{v}_{i_2} & \cdots & \bar{v}_{i_n} \\
	 \mid &  \mid  & \mid  & \mid 
\end{pmatrix}
	 	 =
\begin{pmatrix}
	a_{1i_1} & a_{1i_2} & \cdots & a_{1i_n} \\
	a_{2i_1} & a_{2i_2} & \cdots & a_{2i_n} \\
	\vdots  & \vdots			&	\ddots & \vdots	\\
	a_{ni_1} & a_{ni_2} & \cdots & a_{ni_n}
\end{pmatrix}
\]
has determinant $\pm 1$. By \cite{Oh1983} any effective smooth $\T^n$-action on a simply-connected $(n+2)$-manifold has legal weights. Conversely, given a disk $N^\ast$ equipped with legally weighted orbit data, there is a closed, simply-connected smooth $(n+2)$-manifold $N$ with an effective action of a torus $\T^n$ such that $N^\ast$ is the orbit space of $N$ (see \cite[Section~4]{Oh1983}). We state this fact for future reference.
	 
\begin{thm}[Remark~4.7 in \cite{Oh1983}]\label{t: Weights realized}
For $n\geqslant 2$ and a family of legal weights $\{\bar{v}_{1},\ldots,\bar{v}_{m}\}\subset \Z^n$ with $m\geqslant n$ there exists a closed, simply-connected $(n+2)$-manifold admitting a cohomogeneity two $\T^n$-action that realizes the family $(a_{i1},\ldots,a_{im})$ as weights.
\end{thm}

Let $M$ and $N$ be two closed, simply-connected smooth $(n+2)$-manifolds with effective $\T^n$ actions, with $n\geqslant 2$. Observe that both $M^\ast$ and $N^\ast$  are closed $2$-disk, and thus have a unique smooth structure. A  diffemorphism $f^\ast\colon M^\ast\to N^\ast$ is called \emph{weight preserving} if the isotropy subgroup of the orbit in $N$ corresponding to  $f^\ast(x^\ast)$ equals the isotropy subgroup of the orbit in $M$ corresponding to $x^\ast$. We say that the orbit spaces $M^*$ and $N^*$ are \emph{isomorphic} if there exists a weight-preserving diffeomorphism between them. From the work of Oh given   a weight-preserving diffeomorphism $f^* \colon M^* \to N^\ast$, then there exists an equivariant diffeomorphism $f:M\to N$ which covers $f^*$. More generally, one has the following equivariant classification theorem.


\begin{thm}[Theorem~2.4~\cite{Kim1974}, and Theorem~1.6~\cite{Oh1983}]\label{T:OH_EQUIV}
Two closed, sim\-ply-con\-nected smooth $(n+2)$-manifolds with an effective $T^n$-action are equivariantly diffeomorphic if and only if their orbit spaces are isomorphic. 
\end{thm}


\subsection{Singular Riemannian Foliations}

A \emph{singular Riemannian foliation} on a Riemannian manifold $M$, which we denote by $(M,\fol)$, is the decomposition of $M$ into  a collection $\fol=\{L_p\mid p\in M\}$ of disjoint connected, complete, immersed submanifolds $L_p$, called \emph{leaves}, which may not be of the same dimension, such that the following conditions hold: 
\begin{enumerate}[(i)]
\item Every geodesic meeting one leaf perpendicularly, stays perpendicular to all the leaves it meets.
\item For each point $p\in M$ there exist local smooth vector fields spanning the tangent space of the leaves.
\end{enumerate}
If $(M,\fol)$ satisfies the first condition, then we say that $(M,\fol)$ is a \emph{transnormal system}. If it satisfies the second one, we say that $(M,\fol)$ is a \emph{smooth singular foliation}. When the dimension of the leaves is constant, we say that the foliation is a \emph{regular Riemannian foliation} or just a \emph{Riemannian foliation}. We refer the reader to \cite{Alexandrino2012} for a more in depth discussion of singular Riemannian foliations.

A natural class of examples of singular Riemannian foliations is given by effective actions of groups by isometries. If we only have a compact Lie group $G$ acting smoothly on $M$, the existence of an invariant Riemannian metric, guarantees that we may consider $G$ acting by isometries (see \cite[Theorem~3.65]{Alexandrino}). A singular Riemannian foliation that arises from a group action is called a \emph{homogenous foliation}. We say that a singular Riemannian foliation is \emph{closed} if all the leaves are closed. The \emph{dimension} of a  foliation $\fol$, denoted by $\dim \fol$, is the maximal dimension of the leaves of $\fol$. We call the foliation \emph{trivial} when $\dim \fol = 0$ or $\dim \fol = \dim M$.  In the first case the leaves are collections of points, and in the second case there is only one leaf, the total manifold. The \emph{codimension} of a foliation is,
\[
	\codim(M,\fol) = \dim M -\dim \fol.
\]
The leaves of maximal dimension are called \emph{regular leaves} and the remaining leaves are called \emph{singular leaves}. Since $\fol $ gives a partition of $M$, for each point $p\in M$ there is a unique leaf, which we denote by $L_p$, that contains $p$. We say that $L_p$ is the \emph{leaf through} $p$. The quotient space $M/ \fol$ obtained from the partition of $M$, is the \emph{leaf space} and the quotient map $\pi \colon M \to M/\fol$ is the \emph{leaf projection map}. The topology of $M$ yields a topology on $M/\fol$, namely the quotient topology. With respect to this topology the quotient map is continuous.  We denote the leaf space $M/\fol$ from this point onward by $M^\ast$ as in the homogeneous case. We denote by $S^\ast$ the image $\pi(S)$ of a subset $S\subset M$ under the leaf projection map.

A singular Riemannian foliation $(M,\fol)$ induces a stratification on $M$. For  $k\leqslant \dim \fol$  we define the $k$-dimensional stratum as:
\begin{linenomath}
\[
	\Sigma_{(k)} = \{p\in M \mid  \dim L_p = k\}.
\]
\end{linenomath}
The \emph{regular stratum} $\Sigma_{\reg}=\Sigma_{(\dim \fol)}$ is an open, dense and connected submanifold of $M$ (see \cite[Lemma~2.2.2]{Radeschi2012}). The foliation restricted to the regular stratum yields a Riemannian foliation $(\Sigma_{\reg},\fol)$, and $\Sigma_{\reg}^\ast$ is open and dense in the leaf space $M^\ast$. Furthermore, by  \cite[Proposition~3.7]{Molino}, if $(M,\fol)$ is a singular Riemannian foliation with closed regular leaves, then $\Sigma_{\reg}^\ast$ is an orbifold. 
Note that the foliation is regular if and only if  $\Sigma_{\reg} = M$.


\subsection{Infinitesimal foliation}\label{SS: infinitesimal foliation}

We start by fixing a point $p\in M$, and consider the normal tangent space $\nu_p L_p$ to the leaf $L_p\subset M$. Next we consider  $\varepsilon > 0$ small enough and set $\nu_p^\varepsilon  = \exp_p(\nu_p L_p)\cap B_{\varepsilon}(0)$, with $B_{\varepsilon}(0)$ the ball of radius $\varepsilon$ in $T_p M$. Taking $S_p = \exp_p(\nu_p^\varepsilon L_p)$ we consider the intersection of the leaves of $\fol$ with $S_p$. This induces a foliation $\fol |_{S_p}$ on $S_p$ by setting the leaves of $\fol |_{S_p}$  to be the connected component of the intersection between the leaves of $\fol$ and  $S_p$. This foliation may not be a singular Riemannian foliation with respect to the induced metric of $M$ on $S_p$, i.e the leaves of $\fol |_{S_p}$ may not be equidistant with respect to the induced metric. Nevertheless, by \cite[Proposition~6.5]{Molino}, the pull-back foliation $ \fol^p = \exp_p^{\ast}(\fol |_{S_p})$ is a singular Riemannian foliation on $\nu_p^\varepsilon L_p$ equipped with the Euclidean metric. We call this foliation $(\nu_p^\varepsilon L_p,\fol^p)$ the \emph{infinitesimal foliation at  $p$}. By \cite[Lemma~6.2]{Molino} the infinitesimal foliation is invariant under homotheties that fix the origin. Furthermore the origin $\{0\}\subset \nu_p^\varepsilon L_p$ is a leaf of the infinitesimal foliation. Since the leaves of $\fol^p$ stay at a constant distance from each other, the fact that the origin is a leaf implies that any leaf of $\fol^p$ is at a constant distance from the origin, and thus it is contained in a round sphere centered at the origin. From this last fact it follows that we may  consider the infinitesimal foliation restricted to the unit normal sphere of $\nu_p L_p$, which we denote by  $\Sp^{\bot}_p$, yielding a foliated round sphere $(\Sp^{\bot}_p, \fol^p)$ with respect to the standard round metric of $\Sp^{\bot}_p$ which is also called the infinitesimal foliation. From here on when we say ``infinitesimal foliation" we refer to  $(\Sp^{\bot}_p, \fol^p)$, or equivalently, to $(\nu_p L_p, \fol^p)$, since  $(\nu_p L_p, \fol^p)$ is invariant under homothetic transformations and thus  it can be recovered from $(\Sp^{\bot}_p, \fol^p)$.

For the particular case of a homogeneous singular Riemannian foliation by an action of a compact Lie group $G$, the infinitesimal foliation at a point $p$ is given by taking the connected components of the orbits of the action of the isotropy group $G_p$ on $\Sp^\perp_p$ via the isotropy representation. Therefore, denoting by $G_p^0$ the connected component of $G_p$ containing the identity element, the infinitesimal foliation is given by considering only the action of $G_p^0$ on $\Sp^\perp_p$ given  by the isotropy representation.


\subsection{Holonomy}\label{SS: holonomy and local projections}

Given a leaf $L\subset M$, a point $p\in L$, and  a path $\gamma\colon [0,1]\to L$ starting at $p$, the following theorem gives us a foliated transformation from $\nu_p L$ to the total space $\nu L$ of the normal bundle $\nu L\to L$. 

\begin{thm}[Corollary~1.5 in \cite{Radeschi15}]\label{t: Sliding along a singular leaf}
Let $L$ be a closed leaf of a singular Riemannian foliation $(M,\fol)$, and let $\gamma\colon [0,1]\to L$ be a piecewise smooth curve with $\gamma(0) = p$. Then there is a map $G\colon [0,1] \times \nu_p L\to \nu L$ such that:
\begin{enumerate}[(i)]
\item $G(t,v)\in \nu_{\gamma(t)} L$ for every $(t,v)\in [0,1]\times \nu_p L$.\label{t: Sliding along a singular leaf a}
\item For every $t\in [0,1]$, the restriction $G\colon \{t\}\times \nu_p L\to \nu_{\gamma(t)} L$ is a linear isometry preserving the leaves  of $\nu_L$.\label{t: Sliding along a singular leaf b}
\item For every $s \in \R$, the map $\exp_{\gamma(t)} (sG(t, v))$ belongs to the same leaf as \linebreak $\exp_p(sv)$.\label{t: Sliding along a singular leaf c}
\end{enumerate}

\end{thm}

We denote by $\mathrm{O}(\Sp^\perp_p,\fol^p)$ the group of \emph{foliated isometries}  of the infinitesimal foliation, i.e. all the isometries which preserve the foliation. Thus from Theorem~\ref{t: Sliding along a singular leaf} given a loop $\gamma$ at $p$, we have a foliated linear isometry $G_\gamma\colon \nu_p L\to \nu_p L$ by setting $G_\gamma(v) = G(1,v)$. Therefore, we have a group homeomorphism $\rho\colon\Omega(L,p)\to \mathrm{O}(\Sp^\perp_p,\fol^p)$ from the loop space of $L_p$ at $p$ to the foliated isometries of the infinitesimal foliation by setting $\rho(\gamma) = G_\gamma$.  We note that an isometry in $\mathrm{O}(\Sp^\perp_p,\fol^p)$ may map a leaf to a different leaf. By $\mathrm{O}(\fol^p)$ we denote the foliated isometries which leave the foliation invariant, i.e. the isometries $f\in \mathrm{O}(\Sp^\perp_p,\fol^p)$ such that for any leaf $\mathcal{L}$ of $(\Sp^\perp_p,\fol^p)$s we have $f(\mathcal{L})\subset \mathcal{L}$. There is a natural action of $\mathrm{O}(\Sp^\perp_p,\fol^p)$ on the quotient $\Sp^\perp_p/\fol^p$. The kernel of this action is $\mathrm{O}(\fol^p)$. 

We now show that if two  loops, $\gamma_1$ and $\gamma_2$,  are homotopic, then $G_{\gamma_1}^{-1}\circ G_{\gamma_2}$ is in the kernel of the action of $\mathrm{O}(\Sp^\perp_p,\fol^p)$ on  $\Sp^\perp_p/\fol^p$. Therefore we obtain a group morphism from $\pi_1(L,p)$ to $\mathrm{O}(\Sp^\perp_p,\fol^p)$.

\begin{lem}\label{p: lin lifts respect homotopy}
Let $\gamma_0$ and $\gamma_1$ be two curves in a closed leaf $L$ which are homotopic relative to the end points, with $\gamma_0(0) = p =\gamma_1(0)$, and $\gamma_0(1)=q=\gamma_1(1)$. Then $(G_1)^{-1}\circ G_0\colon \nu_p L\to \nu_q L$ is homotopic to the identity map. Furthermore this map	 takes every leaf of the infinitesimal foliation $\fol^p$ to itself.
\end{lem} 

\begin{proof}
Let $H\colon [0,1]\times I\to L$ be the homotopy between $\gamma_0$ and $\gamma_1$. By applying Whitney's Approximation Theorem (see for example  \cite[Theorem~9.27]{Lee}), we can assume that $H$ is a smooth map. For $s\in I$ fixed we consider the smooth curve $\gamma_s(t) = H(t,s)$. From the compactness of $[0,1]\times I$ we can find a partition $0=t_0<t_1<\cdots t_{N}=1$ of $[0,1]$ such that for any $s\in I$ the curve $\gamma_s$ restricted to $[t_{i-1},t_i]$ is an embedding.  By extending  the vector field $\gamma_s'(t)$ for $t\in[t_{i-1},t_i]$ to $L$, we obtain  smooth vector fields $(V_s)_i$ on $L$. Since the family of curves $\gamma_s$ varies continuously with respect to $s$ by construction, for each $1\leqslant i \leqslant N$ the family of vector fields $(V_s)_i$ varies smoothly with respect to $s$. This implies that, when we consider for each $\gamma_s$, the map $G_s\colon \nu_p L\to \nu_q L$ given by Theorem~\ref{t: Sliding along a singular leaf}, then $G_s$ varies continuously with respect to $s$ (see \cite{Radeschi2016}). Defining  $K(v,s) = (G_s)^{-1}(G_0(v))$ we obtain a homotopy $K\colon \nu_p L\times I \to \nu_p L$ between the identity $Id\colon \nu_p L\to \nu_p L$ and $(G_1)^{-1}\circ G_0\colon \nu_p L\to \nu_p L$. For $v\in \nu_p L$ fixed, we have, from Theorem~\ref{t: Sliding along a singular leaf}~\eqref{t: Sliding along a singular leaf c}, that $\exp_p((G_s)^{-1}( G_0 (v) ))$ lies in the same leaf of $\fol$ as $\exp_p(v)$. Since $K(v,s)$ defines a path between $v$ and $(G_1)^{-1}(G_0(v))$, we have that $(G_1)^{-1}( G_0(v))$ lies in the same leaf $\mathcal{L}_v$ of $\fol^p$ as $v$. Thus $(G_1)^{-1}(G_0(\mathcal{L}_v))\subset \mathcal{L}_v$.
\end{proof}

\begin{prop}\label{p: representation of pi_1 as holonomy}
Let $(M,\fol)$ be a singular Riemannian foliation, $L$ a closed leaf of the foliation and $p\in L$. There is a well defined group morphism,
\begin{linenomath}
\[
	\rho\colon \pi_1(L,p)\to \mathrm{O}(\Sp^\perp_p,\fol^p)/\mathrm{O}(\fol^p),
\]
\end{linenomath}
given by $\rho[\gamma] = [G_\gamma]$.
\end{prop}

\begin{proof}
Given a loop $\gamma_0$, we consider the linear foliated  transformation,\linebreak $G_0\colon \nu_p L\to \nu_p L$,  given by Theorem~\ref{t: Sliding along a singular leaf}, and set 
\[
\rho[\gamma_0] = [G_0]\in \mathrm{O}(\Sp^\perp_p,\fol^p)/ \mathrm{O}(\fol^p).
\]
From Proposition~\ref{p: lin lifts respect homotopy}, if $\gamma_1$ is a loop homotopic to $\gamma_0$, then we have $(G_1)^{-1}\circ G_0 \in \mathrm{O}(\fol^p)$. Therefore $[G_0]=[G_1]$ in $\mathrm{O}(\Sp^\perp_p,\fol^p)/ \mathrm{O}(\fol^p)$.
\end{proof}

For a closed leaf $L$ of a singular Riemannian foliation $(M,\fol)$ we define the \emph{holonomy of the leaf} $L$ as the image $\Gamma_L< \mathrm{O}(\Sp^\perp_p,\fol^p)/\mathrm{O}(\fol^p)$ of $\pi_1(L,p)$ under the morphism $\rho$. When we consider the holonomy of a leaf $L_p$ trough a point $p\in M$, we  denote it by $\Gamma_p$. A regular leaf $L$ is called  a \emph{principal leaf} if the holonomy group is trivial, and \emph{exceptional} otherwise.

Given a fixed point $p\in M$  and a vector $v\in \Sp^{\bot}_p$, set $q= \exp_p(\varepsilon v)$. If $\varepsilon$ is small enough, then $L_q$ is contained a tubular neighborhood of $L_p$, and thus there is a well defined smooth closest-point projection $\proj\colon L_q\to L_p$ which is a submersion, by  \cite[Lemma~6.1]{Molino}. The connected component of the fiber of the map $\proj$ through $q$ can be identified with the leaf $\mathcal{L}_v\in \fol^p$ through $v$. Taking  $\overline{L}_p = \widetilde{L}_p / \proj_\ast (\pi_1(L_q))$, the quotient of the universal cover $\widetilde{L}_p$ of $L_p$, we have a finite cover $\overline{L}_p\to L_p$ such that $\proj\colon L_q\to L_p$ lifts to a fibration
\begin{equation}
	\mathcal{L}_v\to L_q \overset{\xi}{\to} \overline{L}_p.\label{EQ: Leaf Fibration}
\end{equation}

Clearly fibration \eqref{EQ: Leaf Fibration} is a surjective map by construction. The following proposition gives another way of obtaining the covering $\overline{L}$ of $L$, via a subgroup $H$ of the holonomy group $\Gamma_p$.

\begin{prop}\label{R: holonomy via the space of directions}
For $v\in \Sp^{\bot}_p$ with image $v^\ast\in \Sp_{p^\ast}$, set $H$ to be the subgroup of $\Gamma_p$ fixing $v^\ast$. Then, taking $q=\exp_p(v)$, the  finite cover $\overline{L}_p$ of $L_p$ in the fibration $\xi\colon L_q\to \overline{L}_p$  is $\widetilde{L}_p / H$. 
\end{prop}

\begin{proof}
Let $F= \proj^{-1}(p)$ be the fiber of the metric projection $\proj\colon L_q\to L_p$, which may consist of several connected components. The end of the long exact sequence of the fibration looks like
\begin{linenomath}
\[
	\cdots\to \pi_1(F,q)\to \pi_1(L_q,q)\overset{\proj_\ast}{\longrightarrow}\pi_1(L_p,p)\overset{\partial}{\to}\pi_0(F,q)\to 0.
\]
\end{linenomath}
From  exactness, we conclude that $(\proj_\ast)(\pi_1(L_q,q) ) = \kernel (\partial)$. We recall how the map $\partial\colon \pi_1(L_q,q)\to \pi_0(F,q)$ is defined, following a modification of the definitions presented in Hatcher (see Sections~4.1 and~4.2 in \cite{Hatcher}). Let $\lambda_1\colon \D^1\to\Sp^1$ be the map collapsing $\partial\D^1$ to a point. Let $\delta_0\colon \Sp^0\to \D^1$ be the inclusion as the boundary. Consider a loop $\varphi\colon \Sp^1\to L_p$ with base point $p$. By the homotopy lifting property there is a lift $\lambda\colon \D^1\to L_q$ for the map $\varphi\circ\lambda_1\colon \D^1\to L_p$, with $\lambda(0)=q$. Furthermore, by definition, we have that $\varphi\circ\lambda_1\circ\delta_0=\proj\circ\lambda\circ\delta_0$ is constant. Therefore the image of the map $\psi=\lambda\circ\delta_0$ is contained in $F$. Thus we have a map $\psi\colon \Sp^0\to F$. We define $\partial [\varphi]=[\psi]$. Let $\alpha\colon \Sp^1\to L_p$ be a loop in $\proj_\ast(\pi_1(L_q,q) )=\kernel \partial$. Consider $G_t = G\colon\{t\}\times \nu_p L_p\to \nu L_p$, the transformation given by Theorem~\ref{t: Sliding along a singular leaf} corresponding to $\alpha$. Then $\tilde{\alpha}(t) =\exp_{\alpha(t)} (G_t(v))$ is a lift of $\alpha$ in $L_q$. Since $\partial$ does not depend on the choice of a lift we have that $0=\partial[\alpha]=[\tilde{\alpha}]$. It follows that the end point of $\tilde{\alpha}\colon [0,1]\to F$ is in the same connected component of $F$ as $q$. Therefore we have that $G_1\colon \nu_p L_p\to \nu_p L_p$ fixes the infinitesimal leaf $\mathcal{L}_v$ in $\Sp^\bot_p/\fol_p$. Thus $\proj_\ast(\pi_1(L_q,q) )\subset H$. Conversely, if we start with $[\alpha]\in H$, then for the map $G\colon [0,1]\times \nu_p L_p\to \nu L_p$ given by Theorem~\ref{t: Sliding along a singular leaf}, we have that $G_1$ maps the infinitesimal leaf $\mathcal{L}_v$ to itself. By definition this means that $\exp_p(G_1(v))$ is in the same connected component of $F$ as $q=\exp_p(v)$. Theorem~\ref{t: Sliding along a singular leaf}~\eqref{t: Sliding along a singular leaf c} implies that the path $\exp_p(G_t(v) )$ is a path between $q$ and $\exp_p(G_1(v))$ in $F$. Thus we have that $\partial[\alpha]=0$. Therefore we conclude that $\proj_\ast(\pi_1(L_q,q) ) = H$.
\end{proof}

Proposition~\ref{R: holonomy via the space of directions} gives a  way to detect if there is  holonomy for a  closed leaf of $(M,\fol)$. This proposition extends \cite[Remark~2.3]{Galaz-Garcia2015} to general singular Riemannian fibrations.  

If the map $\xi\colon L_q\to \overline{L}_p$ is proper then, it follows from  Ehresmann's fibration Lemma in \cite{Ehresmann} that $\xi$ is a locally trivial fibration. In the particular case when $L_q$ is compact, then the map $\xi$ is proper. Thus for foliations with closed leaves the fibration \eqref{EQ: Leaf Fibration} is a fiber bundle.

\begin{remark}
We note that from Ehresmann's Lemma the fiber bundle given by the projection map $\xi$, may not have as structure group a Lie group, but rather a very large topological group, namely the diffeomorphism group of the fiber, $\Diff (\mathcal{L}_v)$. Although $\Diff (\mathcal{L}_v)$ is in general not a Lie group, it is  a Frobenious group, i.e. the group operations are smooth with respect to a Frobenious atlas (see \cite{GuijarroWalschap2007}).
\end{remark}

If the foliation $\fol$ is given by the action of a compact Lie group $G$, then the  holonomy of an orbit $G(p)$ is given by $G_p/G_p^0$ (see \cite[Section~3.1]{Radeschi15}). For $q\in M$ close to $p$ with $G_q$ a subgroup of  $G_p$,  the fiber bundles given by \eqref{EQ: Leaf Fibration}
are of the form:
\begin{linenomath}
\[
	G^0_p/G_q \to G/G_q  \to G/G^0_p,
\]
\end{linenomath}
where $G^0_p$ is the connected component of the identity of the isotropy group $G_p$, and $G/G^0_p$ is a cover of the orbit $G/G_p$ (see \cite[Example~2.4]{Galaz-Garcia2015}).

%
%
%
%

\subsection{Principal Stratum}

We now study the topology of the principal stratum $M_{\prin}$ of a closed singular Riemannian foliation $(M,\fol)$ on a compact simply-connected Riemannian manifold $M$.

If $L$ is a  principal leaf of a singular Riemannian foliation $(M,\fol)$ on an $n$-dimen\-sional manifold with all regular leaves closed, then \cite[Theorem~2.5]{Moerdijk} and the fact that $\Gamma_p$ is trivial show that $L^\ast$ is  a regular  point of the orbifold $\Sigma_{\reg}^\ast$. With this we can easily see that $M_{\prin}^\ast$  corresponds to the manifold part of the orbifold $\Sigma_{\reg}^\ast$. Thus $M_{\prin}$ is open and dense in $\Sigma_{\reg}$. Since, in general for a singular Riemannian foliation, $\Sigma_{\reg}$ is open and  dense in $M$, then $M_{\prin}^\ast$ is open and dense in $M^\ast$. Furthermore, from the fact the manifold part of an orbifold is connected (see \cite[Lemma~2.3]{Lage2018}, \cite{Yeroshkin2014}, \cite{Faessler2011}) it follows that  the set $M_{\prin}^\ast$ is connected in $M^\ast$. Since it is locally Euclidean, it is path connected.

We collect these observations in the following proposition.

\begin{prop}\label{T: principal part is open dense connected}
Let $(M,\fol)$ be a singular Riemannian foliation  with closed regular leaves.  Then the principal stratum $M_{\prin}$ is open dense in $M$, and $M_{\prin}^\ast$ is connected and path connected.
\end{prop}

Consider  principal leaves,  $L_0$ and $L$, in a closed singular Riemannian foliation $(M,\fol)$ on a compact, simply-connected manifold $M$. Furthermore assume that $M_{\prin}^\ast$ is simply connected. Take any path $\gamma\colon I \to M_{\prin}^\ast$, with $\gamma(0) = L_0^\ast$ and $\gamma(1) = L^\ast$. Recall that $M_{\prin}$ is a  subset of the regular stratum $\Sigma_{\reg}$. For a fixed point $x\in L_0$, consider the unique smooth horizontal lift $\gamma^x\colon I \to M_{\prin}$ of $\gamma$ through $x$ (see \cite[Proposition~1.3.1]{Gromoll}). We define a homeomorphism $h_\gamma\colon L_0\to L$ by setting $h_\gamma(x) = \gamma^x(1)$. 

\begin{cor}\label{C: homeos between principal leafs are homotopic}
Consider a singular Riemannian foliation $(M,\fol)$ with clo\-sed leaves on a compact simply-connected Riemannian manifold. Assume that the principal stratum $M_{\prin}^\ast$ is simply connected. Fix  principal leaves $L_0$ and $L$ of $\fol$, and consider two paths $\gamma_0\colon I\to M_{\prin}^\ast$ and $\gamma_1\colon I\to M_{\prin}^\ast$, connecting $L_0^\ast$ and $L^\ast$. Then the homeomorphism $h_{\gamma_0}$ is homotopic to $h_{\gamma_1}$.
\end{cor}

\begin{proof}
From the hypothesis that $M_{\prin}^\ast$ is simply connected it follows that there is a homotopy from $H\colon I\times I\to M_{\prin}^\ast$ from $\gamma_0$ to $\gamma_1$ fixing the end points $L_0^\ast$ and $L^\ast$. This defines a continuous family of curves $\gamma_s\colon I\to M_{\prin}^\ast$, by setting $\gamma_s(t) = H(t,s)$. We define a homotopy $\widetilde{H}\colon L_0\times I\to L$ by setting $\widetilde{H}(x,s)=\gamma_s^x(1)$.
\end{proof}

\begin{remark}\label{R: Trivial holonomy means trivial connected image of path space in orthogonal group}
Observe that in general if a leaf $L$ has trivial holonomy $\Gamma_L =\{\mathrm{Id}\}$ this does not mean that the map $\rho\colon \Omega(L,p)\to \mathrm{O}(\Sp_p^\perp,\fol^p)$ is trivial. It means that by Lemma~\ref{p: lin lifts respect homotopy} any element in $\rho( \Omega(L,p))\subset \mathrm{O}(\Sp_p^\perp,\fol^p)$ is homotopic to the identity map. 
\end{remark}

\subsection{Slice Theorem}

With the concepts of the infinitesimal foliation at a point $p$ and the representation of the loop space $\Omega(L,p)$ in $\mathrm{O}(\Sp^\perp_p,\fol^p)$ of a closed leaf in \cite{MendesRadeschi2019} the following theorem is proven which describes a tubular neighborhood of a closed leaf of a singular Riemannian foliation:

\begin{thm}[Slice Theorem, Theorem~A in \cite{MendesRadeschi2019}]\label{T: Slice Theorem Foliations}
Let $(M,\fol)$ be a singular Riemannian foliation, and
let $L$ be a closed leaf with infinitesimal foliation $(\nu_p L,\fol^p)$ at a point $p\in L$. Then there is a
group $K_p \subset \mathrm{O}(\nu_p L)$ of foliated isometries of $(\nu_p L,\fol^p)$ and a principal $K_p$-bundle $P$ over
$L$, such that for small enough $\varepsilon > 0$, the $\varepsilon$-tube $U$ around $L$ is foliated diffeomorphic
to $(P \times_{K_p} \D_p^\perp(\varepsilon), P\times_{K_p} \fol^p
)$.
\end{thm}

Here the group $K_p$ is given as: 
\[
K_p = \{G_\gamma\colon \nu_p L\to \nu_p L\mid \gamma \mbox{ piece-wise smooth loop at } p \}.
\]

Since the action $K_p$ on $\nu_pL$  by construction preserves the infinitesimal foliation $\fol^p$, then diagonal action of $K_p$ on $P\times \nu_p L$ via $G_\gamma(p,v) = (p\cdot G_{\gamma^{-1}}, G_\gamma(v))$ preserves the product foliation $\{P\times \mathcal{L}_v\mid v\in \nu_p L\}$. Denote by $\rho\colon P\times \nu_p L\to  (P\times \nu_p L)/K_p$ the quotient map. The leaves of the foliation $P\times_{K_p}\fol^p$ are given by the images of $\rho(P\times \mathcal{L}_v)$, and make a singular Riemannian foliation \cite[Proposition~1.11]{Radeschi-notes}.

\section{\texorpdfstring{$A$}{A}-foliations}\label{S: A-foliations}

Foliations by tori on compact  Riemannian manifolds were introduced in \cite{Galaz-Garcia2015} as generalizations of smooth effective torus actions. Namely an \emph{$A$-foliation} is a foliation where all the leaves are closed, 
and \emph{aspherical}, i.e. for $n>1$ the $n$-th homotopy group of the leaves is trivial. 
By \cite[Corollary~B]{Galaz-Garcia2015}  the principal leaves of an $A$-foliation on a compact, simply-connected, Riemannian manifold  are homeomorphic to tori.

%
%
%
%

\subsection{Homeomorphism type of the  leaves in an \texorpdfstring{$A$}{A}-foliation}

We recall that given two points $p,q\in M$, with $q$ sufficiently  close to $p$ with respect to the metric of $M$,  in the case when $L_q$ is a principal leaf, then there is a fibration:
\begin{linenomath}
\begin{equation}
\mathcal{L}\to L_q\to \overline{L}_p \tag{\ref{EQ: Leaf Fibration}},
\end{equation}
\end{linenomath}
where $\overline{L}_p = \tilde{L}_p / H$ is a finite cover of $L_p$, and $\mathcal{L}$ is a leaf in the infinitesimal foliation $\fol^p$ (see Proposition~\ref{R: holonomy via the space of directions}). Using this description we describe next the topology of the other leaf types in an $A$-foliation.

Recall from Section~\ref{SS: infinitesimal foliation}, that the infinitesimal foliation at $p\in M$, is obtained from the foliated slice $(S_p, \fol_{|S_p})$, where $S_p = \exp_p(\nu^\varepsilon_p L_p)$ for a sufficiently small $\varepsilon>0$. First we consider the case when the leaves of this foliation are connected. In this case the finite covering $\overline{L}_p$ is trivial, i.e. $\overline{L}_p = L_p$. Thus following the proof of \cite[Theorem~3.7]{Galaz-Garcia2015}, we prove the following result:

\begin{prop}\label{P: F->M->N F conn and M torus then N and F torus}
Let $F$, $M$ and $N$ be topological manifolds, with $F$ connected, and let $F\to M\to N$ be a fibration. If $M$ is homeomorphic to a torus, then $F$ and $N$ are tori.
\end{prop}

\begin{proof}
Since $M$ is aspherical, we have from \cite[Theorem~3.7]{Galaz-Garcia2015} that $F$ and $N$ are also aspherical. From the long exact sequence of the fibration we get:
\[
	0\to \pi_1(F)\to \pi_1(M) \to \pi_1(N) \to 0.
\]
Since $\pi_1(M)$ is an abelian, torsion-free, finitely-generated group, and $\pi_1(F)$ is a subgroup of  $\pi_1(M)$, then $\pi_1(F)$ is an abelian, torsion free, finitely generated group. Thus, by classification of finitely generated abelian groups and the Borel conjecture, $F$ is homeomorphic to a torus. Now assume that $\pi_1(N)$ has torsion. Then for some $k\in \Z$, the cyclic group $\Z_k$ acts freely on the contractible manifold $\widetilde{N}$. Therefore it follows that $\widetilde{N}/\Z_k$ is an Eilenberg-MacLane space $K(\Z_k,1)$. This contradicts the fact that $K(\Z_k,1)$ has infinite cohomological dimension. Thus $\pi_1(N)$ is an abelian, torsion-free, finitely generated group. Again by the classification of finitely generated abelian groups and the  Borel conjecture, $N$  is homeomorphic to a torus.
\end{proof}

\begin{cor}\label{P: Non holomony implies torus}
	In an $A$-foliation all leaves with trivial holonomy are homeomorphic to tori.
\end{cor}

In the case when the leaf $L_p$ has non-trivial holonomy we get the following proposition:

\begin{prop}\label{P: Holonomy implies Bieberbach type}
The leaves (of $\dim\neq 4$) with non-trivial holonomy of an $A$-foliation are homeomorphic to Bieberbach manifolds. 
\end{prop}

\begin{proof}
In the case when the leaf $L_p$ has non-trivial holonomy, applying Proposition~\ref{P: F->M->N F conn and M torus then N and F torus} to  fibration \eqref{EQ: Leaf Fibration} we have that the covering $\overline{L}_p$ is homeomorphic to a torus. Thus, applying the long exact sequence of homotopy groups to the fibration $\overline{L}_p\to L_p$ with finite fiber $F$, we get,
\begin{linenomath}
\[
	0\to\pi_1(\overline{L}_p) \to \pi_1(L_p) \to \pi_0(F)\to 0.
\]
\end{linenomath}
Therefore $\pi_1(L_p)$ is a finite extension of $\pi_0(F)$ by $\pi_1(\overline{L}_p)$.  Assume that $\pi_1(L_p)$ is not torsion-free, and recall that since $\overline{L}_p$ is a torus, we have $\tilde{L}_p=\R^n$. Then there exists a finite cyclic subgroup $\Z_k$ acting on the contractible manifold $\tilde{L}_p = \R^n$. As in the proof of Proposition~\ref{P: F->M->N F conn and M torus then N and F torus} this contradicts the fact that the Eilenberg-MacLane space $K(\Z_k,1)$ has infinite cohomological dimension. Since $\pi_1(\overline{L}_p)$ is $\Z^n$ and $F$ is finite, we have that $\pi_1(L_p) = G$ is a crystallographic group (see \cite[Section~6]{FarrellHsiang1983},\cite{AuslanderKuranishi1957},\cite{Zassenhaus1948}). Thus $\pi_1(L_p)$ is a Bieberbach group, since it is a torsion free crystallographic group. By  \cite[Theorem~6.1]{FarrellHsiang1983}, for $n\neq 3,4$, the leaf $L_p$  is homeomorphic to a Bieberbach manifold. Recall that \cite[Theorem~0.7]{KreckLuck2009} states that the  Borel conjecture is true in dimension $3$. Thus in dimension $3$, since $L_p$ has fundamental group isomorphic to a Bieberbach group, $L_p$ is homeomorphic to a Bieberbach manifold.
\end{proof}

\begin{remark}\label{R: New definition of B-foliations}
In \cite[Definition~3.2]{Galaz-Garcia2015} the authors define a $B$-foliation as   an $A$-foliation with all leaves homeomorphic to Bieberbach manifolds. Since a torus is a Bieberbach space, it follows from Propositions~\ref{P: Non holomony implies torus} and~\ref{P: Holonomy implies Bieberbach type} that any $A$-foliation is a $B$-foliation, provided none of the leaves with non-trivial holonomy is $4$-dimensional. Because of this fact, we will not distinguish between $A$-foliations and $B$, in this work.  Nonetheless, we might consider $B$-foliations as those $A$-foliations on a Riemannian manifold $(M,g)$ for which the leaves with the induced Riemannian metric have sectional curvature equal to $0$, i.e. they are flat manifolds with respect to the induced metric. 
\end{remark}

\begin{remark}
The diffeomorphism type of the leaves of an $A$-foliation  may not be unique. If the leaves have trivial holonomy, i.e. are homeomorphic to tori, then for dimensions $k\geqslant 5$, there exist  different smooth structures $\{U_{\alpha_1},\varphi_{\alpha_1}\}$, $\{U_{\alpha_2},\varphi_{\alpha_2}\}$ on the $k$-torus $T^k$, such that $\tau^k_1 = (T^k,U_{\alpha_1},\varphi_{\alpha_1})$ is homeomorphic (as a topological manifold) to $\tau^k_2 = (T^k,U_{\alpha_2},\varphi_{\alpha_2})$, but $\tau^k_1$ is not diffeomorphic to $\tau^k_2$ (see for example \cite{Hsiang}) .

As a concrete example of this exotic phenomena,  we may consider   an exotic sphere $\Sigma^k$  and the standard torus:
\[
\T^k = \underbrace{\Sp^1\times \cdots \times\Sp^1}_{k\mbox{-times} }.
\]
The manifold $\T^k\#\Sigma^k$ is homeomorphic to $\T^k$ but not diffeomorphic to $\T^k$ (see  \cite[Remark~p.~18]{Farrell} and  \cite[Theorem~3]{FarrellOntaneda}).

In the case when the  leaves with the induced metric have sectional curvature less or equal to $0$, then the smooth Borel conjecture holds \cite[Theorem~3]{FarrellOntaneda}, and there is no ambiguity about the smooth structure of the leaves.
\end{remark}

%
%
%
%

\section{\texorpdfstring{$A$}{A}-foliations of codimension 2}\label{S: A-foliations of codimension 2}

We focus in this section to the particular case of $A$-foliations of codimension $2$ on simply-connected manifolds and how they compare to effective torus actions of codimension $2$.

\subsection{Leaf space of an \texorpdfstring{$A$}{A}-foliation of codimension 2.}

In this section we describe the leaf space of an $A$-foliation of codimension $2$ on a simply-connected manifold $(M,\fol)$. There are two possible cases: the foliation has no singular leaves (i.e. it is a regular foliation), or the foliation has singular leaves. In the second case we will be redundant and call it a singular $A$-foliation.

By Theorem~E  in \cite{Galaz-Garcia2015} it follows that, if the $A$-foliation $\fol$ is regular, then $M$ is diffeomorphic to $\Sp^3$ and the foliation $\fol$ is given by a Hopf weighted action. 

In the second case, when the $A$-foliation $\fol$ is singular, then again by Theorem~E \cite{Galaz-Garcia2015} the leaf space $M^\ast$ is a closed $2$-disk. It follows from \cite[Proposition~1.7]{Lytchak2010} that there are no exceptional leaves in $M_{\reg}$, and thus $M_{\reg} = M_{\prin}$. Furthermore, from the fact that the restriction of the foliation to a stratum $\Sigma_{(k)}$ containing only leaves of dimension $k$ is a regular foliation, then  \cite[Theorem~2.15]{Moerdijk} and the fact that $M^\ast = \D^2$ imply that the singular leaves of $\fol$ also have trivial holonomy. From \cite[Theorem~E]{Galaz-Garcia2015} it follows that there are two types of leaves: a \emph{least singular leaf}, homeomorphic to $T^{n-1}$ with infinitesimal foliation given by the homogeneous foliation $(\Sp^2,\Sp^1)$, and the \emph{most singular leaf}, homeomorphic to $T^{n-2}$ with infinitesimal  foliation given by the homogeneous foliation $(\Sp^3,\T^2)$. The singular leaves project via the quotient map onto the boundary of $M^\ast$. Moreover the boundary of $M^\ast$ consists of a union of arcs $\gamma_i$, $i\in\{1,2,\ldots,r\}$. The points in the interior of these arcs correspond to leaves homeomorphic to $T^{n-1}$. The end points, $F_i$, of these arcs correspond to leaves homeomorphic to $T^{n-2}$. (see figure~\ref{F: leaf space A-fol cod=2}).

\begin{figure}[h!]
\centering
\begin{tikzpicture}[scale=2]
		     Vertices
		    \path[coordinate] (0.5,0.866)  coordinate(A) 
		                (90:1) coordinate(B)
		                (-0.5,0.866) coordinate(C)
		                (135:1) coordinate(D)
		                (-1,0) coordinate(E)
		                (-0.866,-0.5) coordinate(F)
		                (225:1) coordinate(G)
		                (0,-1) coordinate(H)
		                (315:1) coordinate(I)
		                (0:1) coordinate(K)
		                (45:1) coordinate(L);

		    	\draw[thick] ([shift=(30:1cm)]0,0) arc (30:190:1cm);
		    	\draw[thick, -] ([shift=(190:1cm)]0,0) arc (190:215:1cm);
		    	\draw[thick] ([shift=(215:1cm)]0,0) arc (215:335:1cm);
		    \draw[very thick, dotted] ([shift=(31:1cm)]0,0) arc (31:-25:1cm);
		    \fill[gray!20]	([shift=(30:1cm)]0,0) arc (30:390:1cm);
		    
		   \draw[color=black] (0.1,0.3) node{$M^\ast$};
		    
		    \draw[color = black] (0.28,0.68) node{$\gamma_r$};
		    \draw[color = black] (111.5:0.83) node{$\gamma_1$};
		    \draw[color = black] (155:0.8) node{$\gamma_2$};
		    \draw[color = black] (-0.3,-0.8) node{$\gamma_i$};
		    \draw[color = black] (0.28,-0.8) node{$\gamma_{i+1}$};
		    \draw[color = black] (90:1.15) node{$F_1$}; 
		    \draw[color = black] (135:1.19) node{$F_2$};
		    \draw[color = black] (225:1.17) node{$F_i$};
		    \draw[color = black] (268:1.15) node{$F_{i+1}$};
		    \draw[color = black] (45:1.15) node{$F_{r}$};

		    \filldraw[black] 
		                     (B) circle (1pt)
		                     (D) circle (1pt)
		                     (E) circle (1pt)
		                     (G) circle (1pt)
		                     (H) circle (1pt)
		                     (I) circle (1pt)
		                     (L) circle (1pt);
		    \draw[color = black] (2.1, 0.3) node{Principal leaves};
		    \draw[stealth-] (0,0) .. controls +(0.5,-0.5) and +(-0.5,0) .. (1.4,0.3);
		
		    \draw[color = black] (-2.2,0.5) node{Leaf $T^{n-2}$};

		    \draw[-stealth] (-1.9,0.4) .. controls +(0.5,-0.5) and +(-0.5,0.2) .. (-1.05,0);
		    \draw[-stealth] (-1.7,0.6) .. controls +(0.3,0.1) .. (-0.75,0.70);
		
		    \draw[color = black] (-2.05, -0.45) node{Leaf $T^{n-1}$};
		    \draw[-stealth] (-2,-0.6) .. controls (-1.7,-1) and (-1,-1.1) .. (-0.4,-0.95);
		     \draw[-stealth] (-2,-0.3) .. controls (-1.3,0) and (-1.5,-0.3) .. (-1,-.3);
		  \end{tikzpicture}
	\caption{Leaf space of $A$-foliation of codimension two.}\label{F: leaf space A-fol cod=2}
\end{figure}
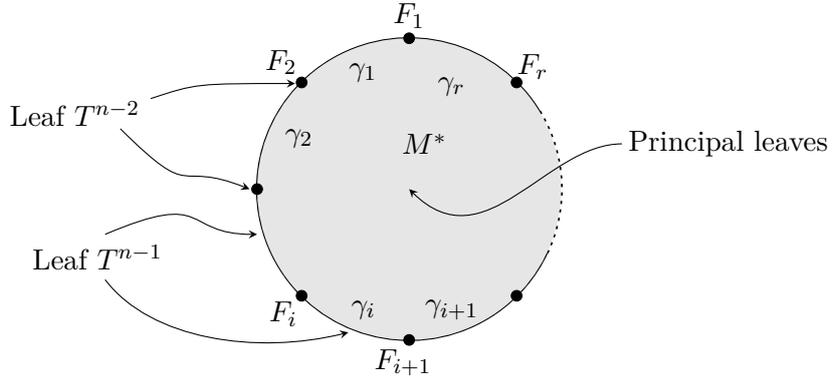

\subsection{Tubular neighborhoods of singular leaves of an \texorpdfstring{$A$}{A}-foliation  of codimension 2}

Before proving our comparison theorem for singular $A$-foliations of codimension $2$ on simply-connected manifolds we describe the tubular neighborhoods of the singular leaves of the foliation.

\begin{thm}\label{T: tubular neighborhood singular leaves of A-fol of codimension $2$}
For an $A$-foliation $(M,\fol)$ of codimension $2$ given a singular leaf $L$  and a fixed point $p\in L$, there exists $\varepsilon>0$ such that the tubular neighborhood of radius $\varepsilon$ around $L$ is foliated diffeomorphic to
\[
	(L\times \D_p^\perp, \{L\times \mathcal{L}_v\mid \mathcal{L}_v\in \fol^p\}).
\]
\end{thm}

\begin{proof}
Recall that the singular leaves of a singular $A$-foliation of codimension $2$ on an $(n+2)$-dimensional simply-connected manifold have dimension $(n-1)$ and $(n-2)$.

Suppose $L$ has dimension $(n-1)$ and fix $p\in L$. Then the infinitesimal foliation $(\Sp^\perp_p,\fol_p)$ is given by the $\Sp^1$-action on $\Sp^2$, which is also the suspension of the circle action on $\Sp^1$. Thus the group $\mathrm{O}(\fol_p)$ is equal to $\mathrm{SO}(2)$. Moreover, since the holonomy is trivial we have that $\mathrm{O}(\Sp_p^\perp,\fol^p) = \mathrm{O}(\fol^p)$. Therefore we see that the action of $K_p$ on $\Sp^\perp_p = \Sp^2$ is homotopic to the trivial action. From this if follows that the associated  disk bundle given by Theorem~\ref{T: Slice Theorem Foliations} is a trivial one.

Now assume that  $L$ has dimension $(n-2)$ and fix $p\in L$. Then the infinitesimal foliation $(\Sp^\perp_p,\fol_p)$ is given by a $\T^2$-action on $\Sp^3$, which decomposes as the join of two circle actions on two copies of $\Sp^1$. That is, $(\Sp^\perp_p,\fol_p) = (\Sp^1,\Sp^1)\star (\Sp^1,\Sp^1)$. Observe that  the group $K_p$ leaves each of the $\Sp^1$ factors invariant. Thus we conclude again that the $K_p$-action on $\Sp^\perp_p = \Sp^3$ is homotopic to a trivial action. Therefore again the associated disk bundle from Theorem~\ref{T: Slice Theorem Foliations} is a trivial one.
\end{proof}

\subsection{Weights of an \texorpdfstring{$A$}{A}-foliation  of codimension 2}

In this section we define invariants called \emph{weights} which allows us to compare $A$-foliations of codimension $2$ on simply-connected manifolds. These invariants extend the notion of weights given for an homogeneous $A$-foliations of dimension $1$, \cite{Fintushel77}, \cite{Orlik}, and of cohomogeneity $2$ (i.e. codimension $2$) \cite{Orlik1970}, \cite{Orlik1974}, \cite{Oh1983}.

Namely, for circle actions on simply-connected $3$-manifolds \cite{Orlik}, or on $4$-manifolds \cite{Fintushel77}, or torus actions of cohomogeneity $2$ on simply-connected $n$-manifolds with $n\geq 4$ \cite{Orlik1970}, \cite{Orlik1974}, \cite{Oh1983}, the weights determine the isotropy group of the orbit, and the isotropy representation. But by the Slice Theorem for group actions, the isotropy group and the isotropy representation describe a small tubular neighborhood of the orbit, as well as the action on this neighborhood.

Moreover, in the case  torus actions of cohomogeneity $2$ on simply-con\-nected $(n+2)$-manifolds, the possible isotropy groups are circles or $2$-torus \cite[(1.1) Lemma]{Oh1983}, and in both cases there is a unique isotropy representation. Thus the tubular neighborhood of the orbits is characterized by the embedding of the circle isotropy groups into $\T^{n}$. This embedding can be represented by an element $\bar{v}\in \pi_1(\T^n)$, which represent the image in $\pi_1(\T^n)$ of the generator of $\pi_1(\Sp^1)$ (see Section~\ref{S: torus actions}). With this alternative description of the weights for homogeneous $A$-foliations of codimension $2$ on simply-connected manifolds, we define an analogous weight for general $A$-foliations of codimension $2$ on simply-connected manifolds. 

Let $(M,\fol)$ be  a simply-connected  $(n+2)$-dimensional manifold with a singular $n$-dimensional $A$-foliation. We fix a principal leaf $L_0\subset M$. We consider any arbitrary point $p\in M$ such that $L_p$ has dimension $(n-1)$. Next we take $v\in \Sp^\perp_p$, such that $q=\exp_p(v)$ is contained in a principal leaf. From  the fact that the regular stratum $M^\ast_{\prin}$ is open and dense in $M^\ast$, there exists  a path $\gamma\colon I\to M_{\prin}^\ast$ connecting $q^\ast$ and $L_0^\ast$. Since $M_{\prin}\to M_{\prin}^\ast$ is a Riemannian submersion, we consider the horizontal lift $\gamma^q$ of $\gamma$ through $q$ (see \cite[Section~1.3]{Gromoll}). We set $q_0=\gamma^q(1)\in L_0$. Recall from subsection~\ref{SS: holonomy and local projections} that, in  this setting, for some cover $\overline{L}_p\to L_p$, we have a fibration
\[
	\mathcal{L}_v\to L_q\to \overline{L}_p. \tag{\ref{EQ: Leaf Fibration}}
\]
Moreover, since for singular $A$-foliations of codimension $2$, all leaves have trivial holonomy it follows that $\overline{L}_p = L_p$.
From \cite[Corollary~B]{Galaz-Garcia2015} and Proposition~\ref{P: F->M->N F conn and M torus then N and F torus},   we have  that $L_q = T^n$, $\overline{L}_p = T^{n-1}$, and $\mathcal{L}_v = T^1$. From the homotopy long exact sequences of the fibration we get a short exact sequence
\begin{linenomath}
\[
	0\to \pi_1(\mathcal{L}_v,v)\to\pi_1(L_q,q)\to \pi_1(\overline{L}_p,p)\to 1. 
\]
\end{linenomath} 
The path $\gamma\colon I \to M_{\prin}^\ast$ connecting $L_0^\ast$ to $L_q^\ast$ induces a homeomorphism $h_\gamma\colon L_0\to L_q$, and thus an isomorphism between $\pi_1(L_q,q)$ and $\pi_1(L_0,q_0)$.

Denote by $e_1$ the  generator of $\pi_1(\mathcal{L}_v,v)=\Z$. It is mapped to an element $\bar{v}_{p}$ in $\pi_1(L_0,q_0) =\Z^n$.

The definition of the integer vector $\bar{v}_{p}$ depends a priori on the choice of path $\gamma$ joining $L^\ast_0$ to $L_q^\ast$. The following lemma shows that in fact, it is independent of the choice of $\gamma$.

\begin{lem}\label{l: weights do not depend on path}
The   element $\bar{v}_p\in \pi_1(L_0,q_0)$ does not depend on the choice of path $\gamma\colon I\to M^\ast$ connecting $L_0$ to $L_q$.
\end{lem}

\begin{proof}
If we choose any other path $\gamma_1$ from $L^\ast_0$ to $L_q^\ast$, then Corollary~\ref{C: homeos between principal leafs are homotopic} shows that the group isomorphisms induced by $(h_{\gamma})_{\ast}\colon \pi_1(L_0,q_0)\to \pi_1(L_q,q)$ and $(h_{\gamma_1})_{\ast}\colon \pi_1(L_0,q_0)\to \pi_1(L_q,q)$ are equal. Therefore the  integer vector $\bar{v}_p$ does not depend of the curve $\gamma$.
\end{proof}

Next we prove that if we choose another vector $w\in \Sp^\perp_p$ such that $\exp_p(w)$ lies in a principal leaf, then we recover the same integer vector $\bar{v}_p$.

\begin{lem}\label{l: weights do not depend on q}
The integer vector $\bar{v}_p$ does not depend on the choice of $v \in\Sp^\perp_p$.
\end{lem}
\begin{proof}
Take $w\in \Sp^\perp_p$  with $w\neq v$, such that $q_1=\exp_p(w)$ lies on a principal leaf $L_{q_1}$. Since $(\Sp^\perp_p, \fol_p)$ is a singular Riemannian foliation with closed leaves, by Proposition~\ref{T: principal part is open dense connected}, the space $(\Sp^\perp_p/\fol^p)_{\prin}$ is path-connected. Therefore there exists a path $\beta\colon I\to (\Sp^\perp_p/\fol^p)_{\prin}$ from $\mathcal{L}_v^\ast\in \Sp^\perp_p/\fol^p$ to $\mathcal{L}_w^\ast\in \Sp^\perp_p/\fol^p$. By taking horizontal lifts of $\beta$ in $(\Sp^\perp_p,\fol^p)_{\reg}$ we obtain a homeomorphism $h_{\beta}\colon \mathcal{L}_v\to \mathcal{L}_w$. By setting $q_1' = \exp_p(h_\beta(v))$, the homeomorphism $h_\beta$ induces an isomorphism $(h_\beta)_\ast\colon \pi_1(\mathcal{L}_v,q)\to \pi_1(\mathcal{L}_w,q_1')$.  From Corollary~\ref{C: homeos between principal leafs are homotopic}, this isomorphism is independent of the choice of $\beta$. 

Let $\sigma$ be a path in $\mathcal{L}_w$ from $q_1$ to $q_1'$. This gives an isomorphism from $\pi_1(\mathcal{L}_w,q_1')$ onto $\pi_1(\mathcal{L}_w,q_1)$, given by mapping an element $[\delta]\in \pi_1(\mathcal{L}_w,q_1')$ to $[\sigma^{-1}\delta\sigma]$. Let $\alpha$ be another path in $\mathcal{L}_w$ from $q_1$ to $q_1'$.  Consider the concatenation of paths $\sigma\alpha^{-1}\delta\alpha\sigma^{-1}$. The path $\alpha\sigma^{-1}$ is a loop based at $q_1'$. Thus we have a conjugation $[\sigma\alpha^{-1}][\delta][\alpha\sigma^{-1}]$ in $\pi_1(\mathcal{L}_v,q_1')$. Since we have an $A$-foliation, $\mathcal{L}_w$ is homeomorphic to a torus. Thus $\pi_1(\mathcal{L}_v,q_1')$ is an abelian group. Therefore the path   $\sigma\alpha^{-1}\delta\alpha\sigma^{-1}$ is homotopic to  $\sigma$, relative to the end points. Thus $\alpha^{-1}\delta\alpha$ is homotopic to $\sigma^{-1}\delta\sigma$. Therefore the isomorphism from  $\pi_1(\mathcal{L}_w,q_1')$ onto $\pi_1(\mathcal{L}_w,q_1)$, does not depend on the path $\sigma$. It follows that  we have a well defined isomorphism from $\pi_1(\mathcal{L}_v,q)$ to $\pi_1(\mathcal{L}_w,q_1)$.

Let $h_\gamma\colon L_q\to L_0$ and $h_\lambda\colon L_{q_1}\to L_0$, be homeomorphisms given by paths $\gamma\colon I\to M^\ast_{\prin}$ and $\lambda\colon I\to M^\ast_{\prin}$. Set $x_0 = h_\lambda(q_1')$, $y_0 = h_\lambda(q_1)$ and $q_0 = h_\gamma(q)$ (see Figure~\ref{F: well def weights}). Denote by $i_1\colon \mathcal{L}_v\to L_q$ and $i_2\colon \mathcal{L}_w\to L_{q_1}$ the inclusions, given by the bundles \eqref{EQ: Leaf Fibration}, of the infinitesimal leaves into the leaves $L_q$ and $L_{q_1}$, respectively. The homeomorphism $h_\beta$ induces an isomorphism from $(h_\gamma\circ i_1)_\ast(\pi_1(\mathcal{L}_v,q))$ onto $(h_\lambda\circ i_2)_\ast(\pi_1(\mathcal{L}_w,q_1'))$. The path $\sigma\colon I\to \mathcal{L}_w$ gives a well defined isomorphism from $(h_\gamma\circ i_2)_\ast(\pi_1(\mathcal{L}_w,q_1') )$ onto $(h_\lambda\circ i_2)_\ast(\pi_1(\mathcal{L}_w,q_1) )$. Thus the generator of $\pi_1(\mathcal{L}_v,q)$ in $\pi_1(L_0,q_0)$ is mapped to the generator of $\pi_1(\mathcal{L}_w,q_1)$. From this we see that the integer 
vector $\bar{v}_p$ does not depend on $v$.
\end{proof}

\begin{figure}[h]
\centering
\includegraphics[scale=0.6]{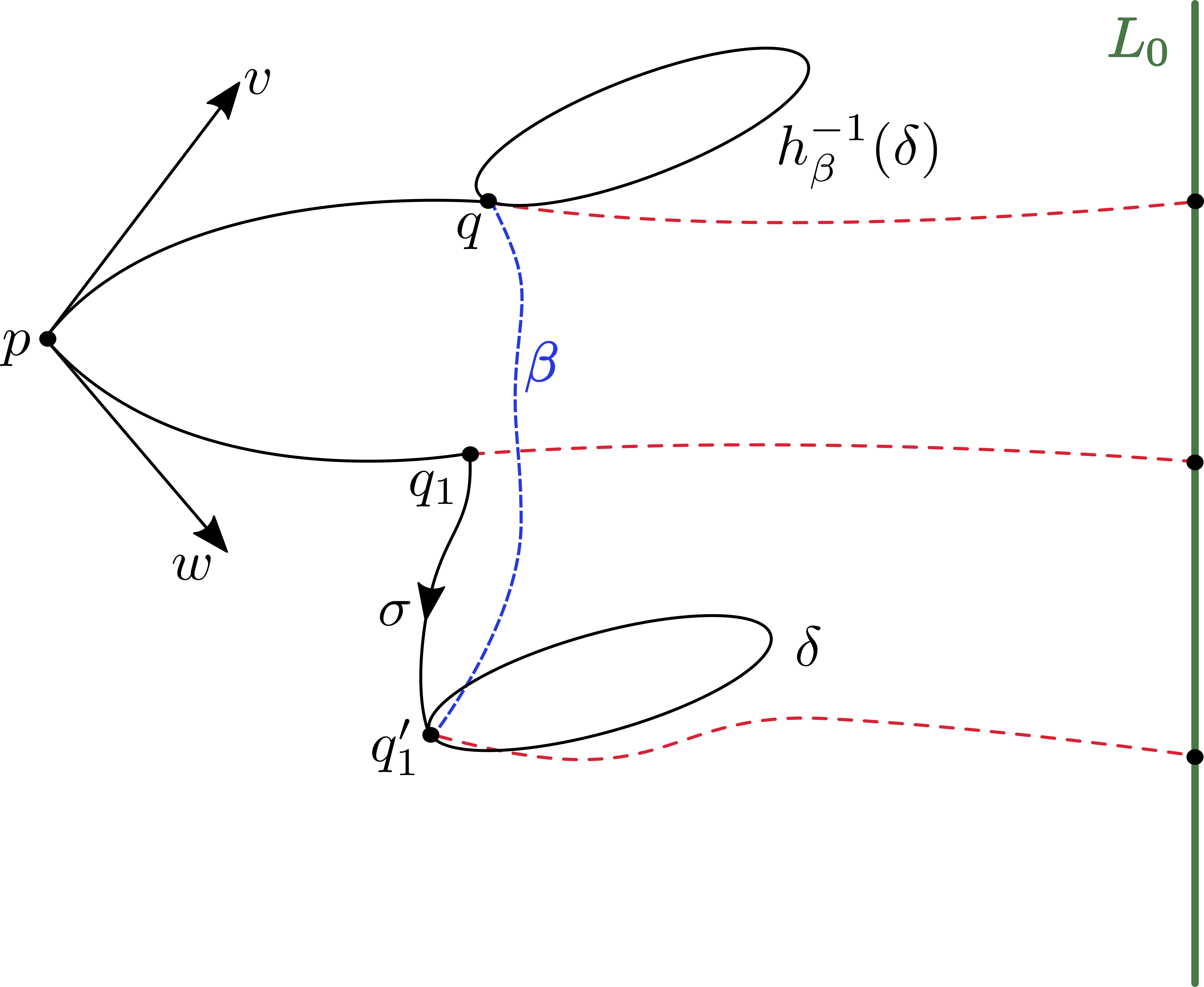}
\caption{Well defined weights.}\label{F: well def weights}
\end{figure}

From the proof of the previous lemma, by using the fact that the fundamental groups of  $L_p$ and $L_0$ are  abelian, it follows that the definition of the  integer vector $\bar{v}_{p}$  does not depend on the choice of the  base point $p$ in $L_p$. We state this explicitly:

\begin{lem}
The weight $\bar{v}_{p}$ of $L_p$ does not depend on the choice of $p\in L_p$.
\end{lem}

Observe that the weights a priori depend of the principal leaf $L_0$ fixed at the beginning. We show that for a singular $A$-foliation of codimension $2$ on a simply-connected manifold the weights do not depend on the choice of the principal leaf $L_0$.

\begin{lem}
Let $(M,\fol)$ be a compact simply-connected $(n+2)$-dimensional manifold with a singular $A$-foliation of codimension $2$. Let $L_1$ and $L_2$ be principal leaves, and take $p\in M$ a point contained in an $(n-1)$-dimensional leaf. Denote by $\bar{v}_{p,1}$ the weight of $L_p$ associated with $L_1$ and by $\bar{v}_{p,2}$ the weight of $L_p$ associated with $L_2$. Then $\bar{v}_{p,1} = \bar{v}_{p,2}$. 
\end{lem}

\begin{proof}
We begin by fixing a third principal leaf $L_0$ and by pointing out that  $\pi\colon M_{\prin}\to M_{\prin}^\ast$ is a fiber bundle, with fibers the principal leaves. Since $M_{\prin}^\ast$ is homeomorphic to an open $2$-disk and thus contractible, then $M_{\prin}$ is foliated homeomorphic to  $M_{\prin}^\ast\times L_0$. Therefore the fundamental group of $L_1$ and $L_2$ get identified with the fundamental group of $L_0$. Moreover the circles in $L_0$ determined by $\bar{v}_{p,1}$ and $\bar{v}_{p,2}$ correspond to the principal leaves of $\fol^p$. Thus we conclude that under the identification $M_{\prin} = M_{\prin}^\ast\times L_0$ we have $\bar{v}_{p,1} = \bar{v}_{p,2}$.
\end{proof}

Last we show that the integer vector $\bar{v}_p$ is constant along the connected components of the singular stratum $\Sigma_{(n-1)}^\ast\in M^\ast$.

\begin{lem}
Let $(M,\fol)$ be an $A$-foliation of codimension $2$ on an $(n+2)$-dimensional simply-connected manifold $M$. Then for $p_1,p_2\in M$  such that $p_1^\ast$ and $p_2^\ast$ are contained in the same connected component of least singular leaves, i.e. $\Sigma_{(n-1)}^\ast\subset M^\ast$, it holds that $\bar{v}_{p_1} = \bar{v}_{p_2}$. 
\end{lem}

\begin{proof}
Recall that a connected component of $\Sigma_{(n-1)}^\ast$ is an open edge contained in $\partial M^\ast$. Thus, there is only one possible path $\gamma^\ast$ joining $p_1^\ast$ to $p_2^\ast$. Since the quotient map $\pi\colon\Sigma_{(n-1)}\to \Sigma_{(n-1)}^\ast$ is a fiber bundle, by taking horizontal lifts of $\gamma^\ast$ we get a homeomorphism $h_{\gamma^\ast}\colon L_{p_1}\to L_{p_2}$. From this we get an isomorphism between $\pi_1(L_{p_1},p_1)$ and $pi_1(L_{p_1},p_1)$. Now we recall that for any close points $q_1 = \exp_{p_1}(v_1)$ and $q_2 = \exp_{p_2}(v_2)$   contained in principal leaves, by Theorem~\ref{T: tubular neighborhood singular leaves of A-fol of codimension $2$} we have a group isomorphism $\pi_1(L_{q_i},q_i) = \pi_1(L_{p_i},p_i)\times \pi_1(\mathcal{L}_{v_i},v_i)$. Here we are abusing the notation by writing as $\pi_1(\mathcal{L}_{v_i},v_i)$ the one-dimensional subgroup of $\pi_1(L_{q_i},q_i)$ generated by the weight $\bar{v}_{p_i}$. With this we conclude that the weights $\bar{v}_{p_1}$ are the same $\bar{v}_{p_2}$. 
\end{proof}

\begin{remark}
Since each connected component of the singular stratum $\Sigma_{(n-1)}^\ast\in M^\ast$ is an open edge in $\partial M^\ast$, we write the collection of weights as follows: we label the edges $\gamma_1,\ldots,\gamma_r$  as shown in Figure~\ref{F: leaf space A-fol cod=2}. To each edge $\gamma_i$ we add the associated weight $\bar{v}_i = (a_{i1},\ldots,a_{in})\in \Z^n$. We collect this information in a list  $\{\bar{v}_1,\ldots,\bar{v}_n\}$ of weights.
\end{remark}

Also we point that the fibration $\eqref{EQ: Leaf Fibration}$ is given by a circle action: 

\begin{prop}\label{C: S1 bundles of A-fol are principal}
The fiber bundle \eqref{EQ: Leaf Fibration} is a principal $\Sp^1$-bundle.
\end{prop}

\begin{proof}
First we show that the bundle \eqref{EQ: Leaf Fibration} is an orientable fiber bundle. We choose an arbitrary orientation for the fiber $\Sp^1_i$ in local charts, to obtain a vector field, tangent to the circles in the total space. Since the $n$-torus is orientable, we can extend this vector field to a basis, such that the transition maps have positive determinant in this basis.

Indeed if  we choose on a local chart an orientation of the fiber $\Sp^1_i\subset T^n$, we can extend it to a basis of the tangent spaces of $T^n$. Since $T^n$ is orientable we can do this construction in such a way that for two open trivial neighborhoods, the orientations of the fibers are positive.

From \cite[Proposition~6.15]{Morita} it follows that \eqref{EQ: Leaf Fibration} is a principal $\Sp^1$-bundle.
\end{proof}

\begin{remark}
For a general  singular Riemannian foliation on a simply-connected manifold, by Theorem \ref{T: Slice Theorem Foliations}, the foliation on a tubular neighborhood of a leaf $L$ is determined by the infinitesimal foliation $(\Sp^\perp_p,\fol^p)$, the conjugation class of a subgroup $K_p<\mathrm{O}(\Sp^\perp_p,\fol^p)$ and a principal $K_p$-bundle over $L$. So these object are the natural choices for ``weights'' in the general setting. In the case of $A$-foliations, we have more structure, due to the leaves having a fixed topology. In the case where a leaf $L$ of an $A$-foliation is regular we have that the infinitesimal foliation is the trivial foliation by points. Thus the foliation on a tubular neighborhood is determined by a subgroup $K_p<\mathrm{O}(\Sp^\perp_p)$. Moreover, by considering the  natural projection $K_p\to \Gamma_L$, and the fact that $\mathrm{O}(\fol^p)= \{\mathrm{Id}\}$ we have that $K_p = \Gamma_L$, and thus it is discrete. This implies that for regular leaves, the foliation in the tubular neighborhood of $L$ is determined by the Biberbach group $\Gamma_L$. Thus the holonomy group $\Gamma_L$ of regular leaves can be consider the weight for such leaves.  For a singular leaf $L$ of an $A$-foliation on a simply-connected manifold, we might consider whether the injection $\pi_1(\mathcal{L}_v,v)\to \pi_1(L_q,q)$, where $\mathcal{L}_v\subset \Sp^\perp_p$ principal leaf of $\fol^p$, together with the holonomy group $\Gamma_L$ determine the foliation in a tubular neighborhood of $L$, as it happens for the codimension $2$ case. 
\end{remark}

%
%
%
%

\subsection{Cross-sections}

With the leaf space description of a singular $A$-foli\-ation of codimension $2$  on a simply-connected manifold at hand, we can show the existence of a cross-section. To prove  this we use results for general singular Riemannian foliations presented in Section \ref{S: obstruction to cross-sections}. Namely, we show that obstructions for the existence of cross-sections presented in Theorems \ref{T: first obstructions} and \ref{T: second obstruction} vanish for $A$-foliations of codimension $2$ on simply-connected manifolds with singular leaves.

\begin{thm}\label{C: Existence cross-sections codimension 2 A foliations}
Let $(M,\fol)$ be a singular $A$-foliation of codimension $2$ on a compact-simply connected manifold of dimension $(n+2)\geqslant 2$. Then there exists a continuous cross-section $\sigma\colon M^\ast\to M$
\end{thm}

\begin{proof}
Observe that a principal $L$ is homeomorphic to an $n$-dimensional torus, and thus it is connected and simple. Since $M_{\prin}^\ast$ is an open $2$-disk, then for $A^\ast\subset M^\ast_{\prin}$ a closed $2$-disk, $A^\ast$ is a closed CW-complex. Consider now  the mapping cylinder $M_{\pi}$ of the quotient map $\pi = \pi|_{A}\colon A\to A^\ast$. Since $A^\ast$ is simply connected it follows that $(M_{\pi},A^\ast)$ is a simple pair. Since $A^\ast$ is contractible,  by Theorem~\ref{T: first obstructions} there exists a cross-section over $A^\ast$ for $\pi\colon A\to A^\ast$.  Also note that $(M^\ast,A^\ast)$ is a CW-pair. 

Now consider the mapping path fibration  $F_\pi\to E_{\pi}\to M^\ast$  of the quotient map $\pi\colon M\to M^\ast$, and recall that $E_\pi$ has the same homotopy groups as $M$. Then from the long exact sequence of the fibration, and the fact that $M$ is simply connected and $M^\ast$ is contractible it follows that $\pi_1(F_\pi)=1$. Thus we conclude that the homotopy fiber $F_\pi$ is a simple space. Since $A^\ast$ is homotopy equivalent to $M^\ast$ it follows from Corollary~\ref{C: M_prin/F homotopy equiv to M/F and section on M_prin imply section on M/F}  that there exists a cross-section $\sigma\colon M^\ast\to M$ extending the cross-section defined over $A^\ast$.
\end{proof}

\begin{remark}
Observe that for a singular $A$-foliation of codimension $2$ on a simply-connected compact manifold since the principal part $M^\ast_{\prin}$ is contractible, then the fiber bundle $\pi\colon M_{\prin}\to M_{\prin}^\ast$ is a trivial bundle. That is, $M_{\prin}$ is diffeomorphic to $M_{\prin}^\ast\times L$. Thus we do not need Theorem~\ref{T: first obstructions} to get a cross-section over a closed disk $A^\ast\subset M_{\prin}^\ast$
\end{remark}

\subsection{Classification of \texorpdfstring{$A$}{A}-foliations of  codimension 2}

We say that two weighted leaf spaces, $M_1^\ast$ and $M_2^\ast$, are \emph{isomorphic} if there is a homeomorphism $\varphi\colon M_1^\ast \to  M_2^\ast$ sending the weights of $M_1^\ast$ to the weights of $M_2^\ast$. The map $\varphi$ is called an \emph{isomorphism between the weighted leaf spaces}, or just simply an \emph{isomorphism between the leaf spaces}.

The following theorem shows the weighted space classifies the topology of $M$ as well as the foliation $\fol$.
\vspace*{1.5ex}
\begin{duplicate}[\ref{T: weights classify the foliation}]
Let $(M_1 , \fol_1 )$ and $(M_2 , \fol_2 )$ be compact, $(n+2)$-dimensional, simply-connected Rie\-mannian manifolds with singular $A$-folia\-tions of codimension $2$.  If the leaf spaces $M_1^\ast$ and $M_2^\ast$ are isomorphic, then $(M_1,\fol_1)$ is foliated homeomorphic to $(M_2,\fol_2)$.
\end{duplicate}
\vspace*{1ex}

\begin{proof}
Fix the two principal leafs $L^1_0\in \fol_1$ and $L^2_0\in\fol_2$ used to define the weights of the foliations, and observe that they are $n$-dimensional. By Corollary~\ref{P: Non holomony implies torus} there exists homeomorphisms $\phi_0^1\colon L_0^1\to T^n$ and $\phi_0^2\colon L_0^2\to T^n$. We set $\phi_0 = (\phi_0^2)^{-1}\circ \phi_0^1$ an homeormphism between $L^1_0$ and $L^2_0$. Denote by $\phi^\ast\colon M_1^\ast\to M_2^\ast$ a weight preserving homeomorphism. Since for $i=1,2$ the sets $(M_i)_{\prin}^\ast$ are open disks, and we have fiber bundles $\pi_i\colon (M_i)_{\prin}\to (M_i)_{\prin}^\ast$ with fibers homemorphic to $L_0^i$, then we conclude that the bundles are trivial ones. Thus, there exists a fiber bundle homeomorphisms $\tau^i_0\colon (M_i)_{\prin}\to  (M_i)_{\prin}^\ast \times L_0^i$. We now set $\phi = \tau_2^{-1}\circ (\phi^\ast\times \phi_0)\circ \tau_1\colon (M_1)_{\prin}\to (M_2)_{\prin}$. Observe that by construction $\phi$ is a foliated homeomorphism. 

We claim that we can extend $\phi$ in a continuous fashion to the singular strata $\pi_1^{-1}(\partial M_1^\ast)$, and next prove this claim. Denote by $\sigma_i\colon M_i^\ast\to M_i$ the cross-sections to the quotient maps $\pi_i\colon M_i\to M_i^\ast$ given by Theorem~\ref{C: Existence cross-sections codimension 2 A foliations}.  Recall that there are only two types of singular leaves: one type has dimension $(n-1)$ and the other one has dimension $(n-2)$, and in both cases the holonomy groups are trivial. Thus for $p^\ast\in M_1^\ast\setminus (M_1^\ast)_{\prin}$  a sufficiently small foliated tubular neighborhood of $L_{\sigma_1(p^\ast)}$ is foliated diffeomorphic via a map $\psi_p$ to $(L_{\sigma_1(p)}\times \D^\perp_{\sigma_1(p)},\{L_{\sigma_1(p)}\times \mathcal{L}_v\mid \mathcal{L}_v\in \fol^{\sigma_1(p)}\})$. Assume that $\dime(L_{\sigma_1(p)}) = n-1$. Then the infinitesimal foliation $(\D^\perp_{\sigma_1(p)},\fol^{\sigma_1(p)})$ is given by the cone of the $\T^1$-action on $\Sp^2$, and the leaves of the infinitesimal foliation correspond to the orbits of the action. Thus any  principal leaf is identified via $\psi_{\sigma_1(p)}$ with $L_{\sigma_1(p)}\times \mathcal{L}_v$, where $\mathcal{L}_v$ is a circle. The space of directions of $M^\ast_1$ at $p^\ast$ is given by $\Sp_p^\perp/\fol^p$ and thus it is isomorphic to $[0,\pi]$. Observe that as we approach one of the boundary points we are shrinking the circle orbit $\mathcal{L}_v$ to a point. Recall that as defined, the weight at $L_{\sigma_1(p)}$ lets us identify the infinitesimal leaf $\mathcal{L}_v$ with a circle in $ L^1_0$ which we also denote by $\mathcal{L}_v$. From this it follows that by our assumption that over $\phi^\ast(p^\ast)$ we have the same weight, then $\phi_0$ maps $\mathcal{L}_v\subset L^1_0$ to the fiber of the fibration $L_0^2\to L_{\phi^\ast(p^\ast)}$. That is $\Phi_0 =\psi_{\phi^\ast(p^\ast)}\circ\phi_0\circ\psi_{p^\ast}^{-1}$ maps the principal leaves of the infinitesimal foliation at $\sigma_1(p^\ast)$ to the principal leaves of the infinitesimal foliation at $\sigma_2\circ\phi^\ast(p^\ast)$. Thus $\Phi_0|_{L_{\sigma_1(p^\ast)}\times \mathcal{L}_{v}}$ factors as a product of homeomorphisms over the principal part of the tubular neighborhood: $\mathrm{Tub}^\varepsilon(L_{\sigma_1(p^\ast)})\cap (M_1)_{\prin}$. From the fact that the missing singular leaves of the infinitesimal foliation are two points we see that we can continuously extend $\phi$ to the stratum of leaves of dimension $(n-1)$.

Assume now that $\dime(L_{\sigma_1(p^\ast)})=n-2$. Then the infinitesimal foliation $(\Sp_{\sigma_1(p^\ast)}^\perp, \fol^{\sigma_1(p^\ast)})$ is given by the linear $\T^2$-action over $\Sp^3$. Thus the space of directions at $p^\ast$ and $\phi^\ast(p^\ast)$ is isometric to $[0,\pi/2]$.  The end points of the space of directions correspond to singular leaves $q_1^\ast$ and $q_2^\ast$ of dimension $(n-1)$. Observe that the fiber of the projection of the principal leaf $L_{\sigma_1(p^\ast)}\times\mathcal{L}_v$ to the singular leaf $L_{\sigma_1(q_j^\ast)}$ is a circle in $\mathcal{L}_v$ corresponding to the circle $\mathcal{L}_{v_j}$ in $L_0^1$ determined by the weight of the edge containing $q_j^\ast$. That is $\mathcal{L}_{v} = \mathcal{L}_{v_1}\times  \mathcal{L}_{v_2}$.  Moreover since the the homeomorphism $\Phi_0$ is weight preserving, it sends the each factor $\mathcal{L}_{v_j}$ in $L^1_0$ to the corresponding circle in $L^2_0$. Thus we conclude that $\Phi_0$ splits as product over the principal leaves $L_{\sigma_1(p^\ast)}\times\mathcal{L}_v = L_{\sigma_1(p^\ast)}\times(\mathcal{L}_{v_1}\times \mathcal{L}_{v_2})$ in  a small tubular neighborhood of $L_{\sigma_1(p^\ast)}$. Therefore we can extend $\phi$ in a continuous way to the strata of $(n-2)$-dimensional leaves, and thus we get a  well defined foliated homeomorphism $\phi\colon (M_1,\fol_1)\to (M_2,\fol_2)$ lifting $\phi^\ast$.
\end{proof}

\subsection{Comparison of \texorpdfstring{$A$}{A}-foliations of codimension two to torus actions}

In this section we give a more detailed analysis of the weights of a singular $A$-folaition of codimension $2$ on a simply-connected manifold and compare them to the weights defined for group actions in \cite{Orlik1970}, \cite{Orlik1974}, \cite{Oh1983}. This allows us to compare singular $A$-foliations of codimension $2$ on simply-connected manifolds to torus actions.

We start by determining the number of weights such a foliation has. From the proof of Theorem~A in \cite{Galaz-Garcia2015}, we are able to determine the number $r$ of different bundles of the form 
\begin{equation}
	\Sp^1 \to T^n \to  T^{n-1},\label{least singular}
\end{equation}
for an $A$-foliation of codimension two on a compact, simply-connected manifold, coming from \eqref{EQ: Leaf Fibration}.

\begin{thm}\label{T: number of edges is greater than n}
Let $(M,\fol)$ be a compact, simply-connected $(n+2)$-manifold with an $A$-foliation of codimension two, and $L_0$ a regular leaf of dimension $n$. If the leaf space $M^\ast$ has $r$-edges in the boundary, then $r\geqslant n$.
\end{thm}

\begin{proof}
We first note that for $A$-foliations of codimension two, a regular leaf is a principal leaf, and fix $p_0\in L_0$. We consider $M_0=M_{\reg}$,  and $B = B(M_0^\ast)$ the Haefliger classifying space of $M_0^\ast$. Recall that there is a fibration $M_0\to B$ with fibers homeomorphic to a principal leaf $L_0$ (see for example \cite[Corollary~5.2]{Alexandrino2007}, \cite[Section~4]{Haefliger84} \cite{FloritGoertschesLytchak2015}, \cite[Section~2.5]{Galaz-Garcia2015}, \cite[Theorem~4.26]{Moerdijk}, \cite[Theorem~10.1]{Tondeur}, \cite[Proposition~2.4]{FloritGoertschesLytchak2015}). Then we obtain the following long exact sequence:
\[
	\cdots\to \pi_2(B,b_0)\to \pi_1(L_0,p_0)\to \pi_1(M_0,p_0)\to \pi_1(B,b_0)\to 1.
\]
By taking $H$ to be the image of $\pi_2(B,b_0)$ under the group morphism \linebreak
$\pi_2(B,b_0)\to \pi_1(L_0,p_0)$, we obtain the following short exact sequence:
\[
	0\to H\to \pi_1(L_0,p_0)\to \pi_1(M_0,p_0)\to \pi_1(B,b_0)\to 1.
\]
Using the fact that for an $A$-foliation of codimension two on a compact, simply-connected manifold, the leaf space is a $2$-disk, we conclude that $H=0$. Consider the fibers of the fibrations given by the codimension $3$ leaves. I.e. we consider the fibers of the fibrations of the from \eqref{least singular}. Observe that by hypothesis, there are $r$ of these fibrations. We consider their homotopy class in $L_0$ and denote by $K$ the subgroup they generate in $\pi_1(L_0,p_0)$. It follows from the proof of in \cite[Theorem~A]{Galaz-Garcia2015}  that $\pi_1(L_0,p_0)$ is generated by $K$ and $H$. Furthermore $K$ splits as an abelian group and a finite $2$ step nilpotent $2$-group. Since by  \cite[Theorem~B]{Galaz-Garcia2015} the leaf $L_0$ is homeomorphic to a torus, we conclude that the finite $2$ step nilpotent $2$-group is trivial. Thus from this discussion it follows that there are at least $n$ fibrations of the form \eqref{least singular}.
\end{proof}

Recalling \cite[Lemma~1.4]{Oh1983}, since the fibers of the fibrations of the form \eqref{least singular} generate a the fundamental group of a principal leaf, we deduce the following property of the weights $(a_{i1},\ldots,a_{in})\in \Z^n$, associated to the least singular leaves.

\begin{lem}\label{L: determinant of n-circles must be one}
For an $A$ foliation of codimension $2$, the determinant of the weights $\{\bar{v}_1\ldots,\bar{v}_m\}$ is $\pm 1$.
\end{lem}

Now we are ready to prove the main result for this section.

\begin{thm}\label{t: Fol homeom to homogen}
Let  $(M,\fol_1)$ be  closed, simply-connected $(n+2)$-manifold  with an $A$-foliation of codimension $2$ and $n\geqslant 2$. Then there exist a closed, simply-connected $(n+2)$-manifold $(N,\fol_2)$ with a homogeneous $A$-foliation of codimension $2$ (i.e. with an effective smooth torus action of cohomogeneity $2$), such that $(N,\fol_2)$ is foliated homeomorphic to $(M,\fol_1)$.
\end{thm}

\begin{proof}
By Lemma~\ref{L: determinant of n-circles must be one}, for an $A$-foliation of codimension two on a closed, simply-connected $(n+2)$-manifold $M$, the weights $(a_{i1},\ldots,a_{in})$ are legal weights in the sense of Oh (see \cite{Oh1983}). Thus by  Theorem~\ref{t: Weights realized}  there is a closed, simply-connected, $(n+2)$-manifold $N$ together with a $\T^n$-action realizing the weights. By Theorem~\ref{T: weights classify the foliation}, the manifolds $M$ and $N$ are foliated homeomorphic.
\end{proof}

\section{Smooth structure of the leaves of an \texorpdfstring{$A$}{A}-foliation of codimension two}\label{S: smooth structure of leaves}

In order to be able to prove Theorem~\ref{T: Codimn 2 A foliation is diffeo to homogeneous fol} we need to study the smooth structure of the leaves of a singular $A$-foliation $(M,\fol)$ of codimension two, on a compact, simply-connected Riemannian $(n+2)$-manifold $M$. We  also show that we can find a smooth cross-section $\sigma\colon M^\ast\to M$ for the quotient map. With these two remarks we are able to strengthen the conclusion in Theorem~\ref{t: Fol homeom to homogen} from foliated homeomorphism to foliated diffeomorphism.

\subsection{Smooth structure of leaves}

Recall that for  an $A$-foliation $(M,\fol)$ of codimension two on a simply-connected  $(n+2)$-manifold,  the least singular leaves are singular leaves of codimension $3$ in $M$, homeomorphic to an $(n-1)$-torus. The most singular leaves of $(M,\fol)$ are singular leaves of codimension $4$ in $M$, homeomorphic to an $(n-2)$-torus. 

We also recall,  that when we fix a principal leaf $L_0$, if we denote a singular leaf by $L_i$, we have a smooth fiber bundle:
\begin{equation}\label{EQ: fiber bundle given by tubular neighborhood}
	\mathcal{L}_i\to L_0\to L_i
\end{equation}
with $\mathcal{L}_i$ diffeomorphic to $T^1$, if $L_i$ is a least singular leaf;  or $T^2$, if $L_2$ is a most singular leaf. Furthermore, for the least singular leaf case, this bundle is a principal circle bundle. Thus for each edge of $M^\ast$ we have a circle action on $L_0$, $\mu_i\colon \T^1\times L_0\to L_0$. 

\begin{lem}\label{L: circle commute}
The circle actions $\mu_i\colon \T^1\times L_0\to L_0$ commute.
\end{lem}

\begin{proof}
We recall that the fiber bundle \eqref{EQ: fiber bundle given by tubular neighborhood} arises via  the intersection of the foliation $\fol$ with the normal space of the leaf $L_1$.

In the particular case for a homogeneous foliation $(N, \T^n)$, the principal circle bundles \eqref{EQ: fiber bundle given by tubular neighborhood} are given by circle subgroups of the principal leaf $\T^n$. This implies in the homogeneous case, that each of these  bundles are trivial, and thus admit a cross-section.

It follows from Theorem~\ref{t: Fol homeom to homogen} that  the given $A$-foliation $(M,\fol)$   is foliated homeomorphic to a homogeneous foliation $(N,\T^n)$, via a homeomorphism $\phi$. In particular, a tubular neighborhood of a singular leaf $L_1$ of $(M,\fol)$ is foliated homeomorphic to a tubular neighborhood of a singular orbit in $N$, via $\phi$. This implies that the foliated homeomorphism induces an fiber bundle isomorphism between the bundles $\mathcal{L}_i\to L_0\to L_i$ and the  bundles $\T^1\to \T^n\to \T^n/\T^1$. Thus, for each fiber bundle \eqref{EQ: fiber bundle given by tubular neighborhood} the homeomorphism $\phi$ makes the following diagram commute:
\begin{center}
\begin{tikzcd}
	L_0 \arrow[d] \arrow[r, "\phi"] & \T^n\arrow[d]\\
	L_i \arrow[r,"\phi"] & \T^{n-1}
\end{tikzcd}
\end{center}
From this, it follows that the cross-section for the bundle $\T^n\to \T^{n-1}$ gives rise to a cross-section for the bundle $L_0\to L_i$. This implies that the bundles \eqref{EQ: fiber bundle given by tubular neighborhood} are trivial for each $i$. In particular, from this it turns out that $\phi\colon L_0\to \T^n$ is equivariant with respect to each action $\mu_i$. 
 
Observe that for the homogeneous case, the circle actions on the principal leaf $\T^n$ commute. Then from the fact that $\phi$ is equivariant, it follows  that the actions $\mu_i$ commute.
\end{proof}

From the previous lemma, plus the fact that the actions are smooth, we obtain the following corollary.

\begin{cor}
Let $(M,\fol)$ be an $A$-foliation of codimension two on a compact simply-connected $(n+2)$-manifold. Then the principal leaf $L_0$ is diffeomorphic to the standard torus $\T^n$.
\end{cor}

\begin{proof}
We show that there exists a free smooth $\T^n$-action on the principal leaf $L_0$ of the foliation. Thus for $n\geq 2$, the principal leaf of an $A$-foliation $(M,\fol)$ of codimension two on a compact, simply-connected  $(n+2)$-manifold is diffeomorphic to the standard torus $\T^n$.

From the homogeneous case we know that there exists a set of indices $\{i_1,i_2,\ldots,i_n\}$ such that circle subgroups, given by the fibrations \eqref{EQ: fiber bundle given by tubular neighborhood} defined by the indexes $i_\ell$ generate the principal torus $\T^n$. Using the foliated homeomorphism $\phi$ given by Theorem~\ref{t: Fol homeom to homogen}, and Lemma~\ref{L: circle commute}, we know  we  define the $\T^n$-action  $\mu\colon \T^n \times L_0 \to L_0$ on the principal leaf $L_0$ as
\begin{linenomath}
\[
\mu((\xi_{1},\ldots,\xi_n),p) = \mu_{i_1}(\xi_{1},\mu_{i_2}(\xi_{2}, \cdots ,\mu_{i_n}(\xi_{n}(p))\cdots)).
\]
\end{linenomath}
Since the actions $\mu_{i_j}$ commute $\mu$ gives a continuous action of the standard $n$-torus, $\T^n$, on the principal leaf $T^n$. Furthermore, the action $\mu$ is free and smooth since each of the transformations $\mu_{i_j}$  are free and smooth.
\end{proof}

\begin{remark}
We note that we  have exactly $r$ of these bundles. One for each edge in $\partial M^\ast$. The index $i$ on the fiber is added to be able to distinguish the edge we are referring to. 
\end{remark}

\begin{remark}[Dimension $6$]
Consider  an $A$ foliation $(M,\fol)$ of codimension two on a simply-connected $n$-manifold, with $n\leq 5$. Let $L_0$ be a fixed principal leaf. Recall that  the leaf $L_0$ is homeomorphic to an $(n-2)$-dimensional torus.  Since in this case all leaves have dimension less than $3$, then the smooth structure is unique. So we get the conclusions of Theorem~\ref{T: Codimn 2 A foliation is diffeo to homogeneous fol}.
\end{remark}

\begin{remark}
For the case when $M$ is of dimension $6$, we recall that, from a tubular neighborhood of a most singular leaf $L_p$ we have a smooth $2$-torus bundle over $L_p$ with total space $L_0$. In this case $L_0$ is homeomorphic to a $4$-torus. These bundles were classified in \cite{Fukuhara83}, and Ue showed in \cite{Ue90} that they admit a geometric structure in the sense of Thurston (c.f.\ \cite{Scott1983}). From the explicit list given in \cite{Geiges92}, we see the total space $L_0$ admits an Euclidean geometry. This implies that the leaf $L_0$ admits a flat Riemannian metric (possibly different from the one given by $M$). It follows from Theorem~3 in \cite{FarrellOntaneda} that $L_0$ is diffeomorphic to the standard $4$-torus. The other  leaves have dimension less or equal to $3$, and thus a unique smooth structure. With these observations we have an alternative proof of Theorem~\ref{T: Codimn 2 A foliation is diffeo to homogeneous fol}.
\end{remark}

Now we prove that  also the singular leaves of an $A$-foliation of codimension two on a compact, simply-connected manifold are diffeomorphic to standard tori. 

\begin{cor}\label{C: least singular leaf is standard}
The least singular leaf of an $A$-foliation $(M,\fol)$ of codimension two on a compact, simply-connected $(n+2)$-manifold $M$ is diffeomorphic to the standard torus.
\end{cor}
\begin{proof}
For the least singular leaf $L_{p_i}$ the claim follows from the fact that the fiber bundle \eqref{EQ: Leaf Fibration} is an $\Sp^1_i$-principal bundle, combined with the fact that the total space is the standard torus $\T^n$. Thus the least singular leaf $L_{p_i}$ is diffeomorphic to $\T^n / \Sp^1_i = \T^{n-1}$, i.e. the standard $(n-1)$-dimensional torus. 
\end{proof}

We recall that, if  $x_i$ is a point in $(M,\fol)$, so that $L_{x_i}^\ast$ is a vertex in $M^\ast$, then $L_{x_i}$ is a most singular leaf of $\fol$, and it is homeomorphic to $T^{n-1}$. Furthermore we can choose $p_i$ close enough to $x_i$ in $M$, such that $L_{p_i}$ is a least singular leaf. We point out that the leaf $L_{p_i}$ has trivial holonomy.  Thus for the leaves $L_{p_i}$ and $L_{x_i}$  fibration \eqref{EQ: Leaf Fibration} is a fiber bundle of the form:
\begin{linenomath}
\begin{equation}\label{E: least sing to most sing circle bundle}
	\Sp^1\hookrightarrow T^{n-1}\to T^{n-2}. 
\end{equation}
\end{linenomath}
By  the same arguments given for  the proof of Proposition~\ref{C: S1 bundles of A-fol are principal} we can prove that this bundle  is a principal $\Sp^1$-bundle. With theses remarks we prove the following proposition:

\begin{prop}
The most singular leaf of an $A$-foliation $(M,\fol)$ of codimension two on a compact, simply-connected $(n+2)$-manifold $M$ is diffeomorphic to the standard torus.
\end{prop}

\begin{proof}
Since \eqref{E: least sing to most sing circle bundle} is a  principal $\Sp^1$-bundle we conclude that $L_{x_i}$ is diffeomorphic to $L_{p_i} / \Sp^1$. From Corollary~\ref{C: least singular leaf is standard}, we have that the least singular leaf $L_{p_i}$ is diffeomorphic to $\T^{n-1}$. Thus the most singular leaf is diffeomorphic to  $\T^{n-1} /\Sp^1 = \T^{n-2}$. 
\end{proof}

\subsection{Smooth cross-section}

We show that for an $A$-foliation $(M,\fol)$ of codimension two on a compact, simply-connected, Riemannian manifold $M$, the quotient map $M\to M^\ast$ is smooth.

\begin{lem}\label{L: Least singular leaf has smooth  infinitesimal foliation projection}
Consider an $A$-foliation $(M,\fol)$ of codimension two on a compact, simply-connected manifold. Let $p\in M$ be such that $L_p$ is a least singular leaf (i.e. $L_p$ has codimension $3$ in $M$). Then the following hold  for the infinitesimal foliation $(\Sp_p^\perp,\fol^p)$ at $p$.
\begin{enumerate}[(i)]
\item The quotient space $\Sp_p^\perp/\fol^p$ is homeomorphic to the closed interval $[0,\pi]$, and thus it admits a unique smooth structure.\label{L: Least singular leaf has smooth  infinitesimal foliation projection I}
\item The quotient map $\Sp_p^\perp \to  \Sp_p^\perp/\fol^p$ is smooth.\label{L: Least singular leaf has smooth  infinitesimal foliation projection II} 
\end{enumerate}
\end{lem}

\begin{proof}
We note that $(\Sp^\perp_p,\fol^p)$ is an $A$-foliation of codimension $1$, with principal leaf homeomorphic to $\Sp^1$. It follows from \cite[Theorem~D]{Galaz-Garcia2015} that $(\Sp^\perp_p,\fol^p)$ is the homogeneous foliation $(\Sp^2,\Sp^1)$. Furthermore, from \cite{Galaz-Garcia2018} and \cite{Mostert57} it follows that,  any smooth action of $\Sp^1$ on $\Sp^2$ is equivalent (i.e. there exists an equivariant diffeomorphism) to the linear $\Sp^1$ action on $\Sp^2$. Thus the quotient map is smooth. 
\end{proof}

\begin{lem}\label{L: Most singular leaf has smooth  infinitesimal foliation projection}
Consider an $A$-foliation $(M,\fol)$ of codimension two on a compact, simply-connected manifold. Let $p\in M$ be such that $L_p$ is a most singular leaf (i.e. $L_p$ has codimension $4$ in $M$). Then the following hold  for the infinitesimal foliation $(\Sp_p^\perp,\fol^p)$ at $p$.
\begin{enumerate}[(i)]
\item The quotient space $\Sp_p^\perp/\fol^p$ is homeomorphic to the closed interval \linebreak$[0,\pi/2]$, and thus it admits a unique smooth structure.\label{L: Most singular leaf has smooth  infinitesimal foliation projection I}
\item The quotient map $\Sp_p^\perp \to  \Sp_p^\perp/\fol^p$ is smooth.\label{L: Most singular leaf has smooth  infinitesimal foliation projection II} 
\end{enumerate}
\end{lem}

\begin{proof}
We note that $(\Sp^\perp_p,\fol^p)$ is an $A$-foliation of codimension $1$, with principal leaf homeomorphic to $\T^2$. It follows from  \cite[Theorem~D]{Galaz-Garcia2015} that $(\Sp^\perp_p,\fol^p)$ is the homogeneous foliation $(\Sp^3,\T^2)$, given by the standard linear action. 

We consider $\Sp^3$ as the unit sphere in $\C^2$ and we  use  the so-called \emph{Hopf coordinates} for $\Sp^3$,  given by $(\theta_1,\theta_2,\eta)\mapsto (\sin\eta e^{i\theta_1},\sin\eta e^{i\theta_2},\cos\eta)$
with $\theta_1\in [0,2\pi]$,  $\theta_2\in [0,2\pi]$, and $\eta\in [0,\pi/2]$. We parametrize the $2$-torus $\T^2 =\Sp^1\times \Sp^1$ by the angles $(\alpha,\beta)$. With these coordinates, the action of $\T^2$ on $\Sp^3$ is given by:
\[
	(\alpha,\beta)(\theta_1,\theta_2,\eta) = (\theta_1+\alpha,\theta_2+\beta,\eta).
\]
Thus the quotient map $\Sp^3 \to  \Sp^3/\T^2$ is given by $(\theta_1,\theta_1,\eta)\mapsto \eta$. This proves both claims.
\end{proof}

\begin{prop}\label{R: smooth cross-sections for codim 2 A fol}
Let  $(M,\fol)$ be an $A$-foliation of codimension two on a compact, simply-connected manifold. Then  the leaf space $M^\ast$ admits a unique smooth structure. Furthermore there is a smooth cross-section $\overline{\sigma}\colon M^\ast \to M$ with respect to this smooth structure.
\end{prop}

\begin{proof}
Recall that the leaf space $M^*$ of an $A$-foliation $(M,\fol)$ of codimension two on a closed, simply-connected manifold is homeomorphic to a  $2$-disk. Thus $M^*$ carries a unique smooth structure proving the first claim of the proposition. In this case in we get a smooth cross-section $\overline{\sigma}\colon M^\ast\to M$ as follows. Let $\sigma\colon M^\ast\to M$  be a  cross-section obtained from Theorem~\ref{C: Existence cross-sections codimension 2 A foliations}.  By Lemma~\ref{L: Least singular leaf has smooth  infinitesimal foliation projection} and Lemma~\ref{L: Most singular leaf has smooth  infinitesimal foliation projection}, for each infinitesimal foliation, the quotient map $\Sp^\perp_p\to \Sp^\perp_p/\fol^p$ is smooth. Since for any point $p\in M$, the leaf $L_p$ has trivial holonomy group, a local neighborhood of $p^\ast$ is given by a cone over $\Sp^\perp_p/\fol^p$. This implies that the quotient map $\pi\colon M\to M^\ast$ is smooth. 

It follows from \cite[Theorem~3.3]{Hirsch}  that the space of smooth functions \linebreak$C^\infty (M^\ast, M)$ is dense in the space of continuous functions $C^0 (M^\ast, M)$ with respect to the strong topology. Therefore given  a cross-section $\sigma\colon M^\ast \to M$, there exists a smooth map $h: M^\ast \to M$ close to $\sigma$ in $C^0(M^\ast, M)$. Since the quotient map $\pi\colon M \to M^\ast$ is smooth, then the map $\overline{\sigma}\colon M^\ast \to M$, defined as $\overline{\sigma} = h\circ (\pi \circ h)^{-1}$ is smooth. By construction the map $\overline{\sigma}$ is a cross-section for the map $\pi\colon M \to M^\ast$.
\end{proof}

We end the present work with the proof of the following lemma, which yields a proof of Theorem~\ref{T: Codimn 2 A foliation is diffeo to homogeneous fol}.

\begin{lem}\label{R: A-foliated homeo can be smoothed if sections smooth and no exotic structure}
Let $(M_1,\fol_1)$ and $(M_2,\fol)$ be compact, simply-connected manifolds, with $A$-foliations with standard diffeomorphism type, and isometric leaf spaces. If the leaf spaces $M_1^\ast$ and $M_2^\ast$ are homeomorphic to smooth manifolds, there exists a smooth weight isomorphism $f^\ast M_1^\ast \to M^\ast_2$, and the cross-sections $\sigma_i\colon M_i^\ast\to M_i$ are smooth with respect to these smooth structure, then the foliated homeomorphism of Theorem~\ref{T: weights classify the foliation} is a foliated diffeomorphism.
\end{lem}

\begin{proof}
The smoothness  follows from the following two observations. First we note that that the foliated  homeomorphism of Theorem~\ref{T: weights classify the foliation} is  defined by the  composition of the map $f^\ast$ and the cross-sections $\sigma_i$, which are smooth by hypothesis. Second, the  fact that the leaves are diffeomorphic to $\R^n/\Gamma$ is used to show that the dependency of $x\in L$ with respect to this center is smooth, once we have chosen our center of the Dirichlet domain $y =\sigma_i(x^\ast)$.
\end{proof}

\begin{proof}[Proof of Theorem~\ref{T: Codimn 2 A foliation is diffeo to homogeneous fol}]
If $n=3$ then the principal leaf is $\Sp^1$. By \cite[Theorem~3.11]{Galaz-Garcia2015} the conclusions of Theorem~\ref{T: Codimn 2 A foliation is diffeo to homogeneous fol} follow.
Now we consider the case $n\geqslant 4$. In this case  the leaf space of an $A$-foliation $(M,\fol_1)$ of codimension two on a compact, simply-connected, Riemannian manifold $M$, is a $2$-disk. Then $M^\ast$ admits a smooth structure in a unique way. Furthermore, by Theorem~\ref{t: Weights realized}, there exists a closed Riemannian manifold $N$ with an homogeneous foliation $\fol_2$. This foliation has  the property  that $N^\ast$ is weighted diffeomorphic to $M^\ast$. Last we remark that, by Proposition~\ref{R: smooth cross-sections for codim 2 A fol}, the hypotheses of Lemma~\ref{R: A-foliated homeo can be smoothed if sections smooth and no exotic structure} are satisfied, and thus the result follows.
\end{proof}

\section{Obstruction Theory for cross-section of singular Riemannian foliations}\label{S: obstruction to cross-sections}

In this section we present  Obstruction Theory and apply it to the setting of general singular Riemannian foliations, stating general sufficient conditions for the existence of a cross-section $\sigma\colon M^\ast\to M$ for the quotient map $\pi\colon M\to M^\ast$ of a singular Riemannian foliation $(M,\fol)$ on a simply-connected manifold. We begin presenting the results for general singular Riemannian foliations obtained from the application of Obstruction Theory. In the last section we present the general Obstruction Theory for fibrations between CW-complexes.

\subsection{Obstruction Theory applied to singular Riemannian foliations}\label{S: obstruction theory applied to SRF}

In this sections we focus ourselves to a a given singular Riemannian foliation $(M,\fol)$ on a simply-connected manifold, and  apply  Obstruction Theory presented in Section~\ref{S: general Obstruction Theory} to the quotient map $\pi\colon M\to M^\ast$ to give sufficient conditions for the existence of a cross-section $\sigma\colon M^\ast\to M$ for $\pi$. We begin by stating some topological definitions we need in our results.

Given a connected topological space $W$ and a connected subspace $A\subset W$, fixing a point $w_0 \in A$, recall that there is an action of $\pi_1(A,w_0)$ on $\pi_k(W,A,w_0)$ \cite[p.~345]{Hatcher}. Since $A$ and $W$ are connected this actions does not depend on the base point. We say that a pair $(W,A)$ of connected spaces is an \emph{$n$-simple pair} if this action is trivial for all $k\leqslant n$, and the pair is a \emph{simple pair} if it is $n$-simple for all $n\geqslant 1$. Also observe that there is an action of $\pi_1(W,w_0)$ on the homotopy groups $\pi_k(W,w_0)$. We say that $W$ is  \emph{$n$-simple} if this action is trivial for all $k\leqslant n$. We say $W$ is \emph{simple} if it is $n$-simple for all $n\geqslant 1$, see \cite[Section~4.1]{Hatcher}. 

Now we recall the basic topological constructions of the mapping path fibration and the mapping cylinder of a continuous map $f\colon X\to Y$. Denote by $Y^I$ the space of all continuous paths $\gamma\colon I\to Y$ equipped with the compact-open topology. There is a fibration $q\colon Y^I\to Y$, defined as $q(\gamma) = \gamma(0)$. Then we define the \emph{mapping path fibration of $f$}, which we  denote by $\pi_f\colon E_{\pi_f}\to Y$, as follows: the total space is  $E_{\pi_f} = \{(x,\gamma)\in X\times Y^I\mid f(x)=\gamma(1)\}$, and the projection is defined as $\pi_{f}(x,\gamma) = \gamma(0)$. The space $E_{\pi_f}$ is homotopy equivalent  to $X$, and the map $\pi_f$ is a fibration, with fiber $F_f$ called the \emph{homotopy fiber of $f$}. The \emph{mapping cylinder of $f$} is the space $M_f$ defined as follows: we consider $(X\times I)\sqcup Y$ and identify $(x,1)$ with $f(x)$. There is a natural inclusion map $i\colon X\to M_f$ given by $i(x) = [x,0]$. Then $M_f$ is homotopy  equivalent to $Y$ and the map $i$ is a cofibration.

Recall that given  a closed singular Riemannian foliation $(M,\fol)$, when we consider a closed subset $X\subset M_{\prin}$, consisting of principal leafs the projection map $M\to M^\ast$ restricted to $X$ yields a fiber bundle:
\begin{linenomath}
\[
	L\to X \to X^\ast.
\]
\end{linenomath}
We apply Theorem~\ref{T: Obstructions to extension} to get a family of obstructions, which we will call \emph{first obstructions}.

\begin{thm}\label{T: first obstructions}
Let $(M,\fol)$ be a closed singular Riemannian foliation with $M$ simply-con\-nected, quotient map $\pi\colon M\to M^\ast$, and principal leaf $L$, which is simple and connected. Consider $X\subset M_{\prin}$  such that $X^\ast$ is a CW-complex which is simply-connected. Denote by $M_\pi$ the mapping cylinder of $\pi\colon X\to X^\ast$ and assume that $(M_\pi,  X)$ is simple. 
Then there is a family of obstructions $\omega^1_k \in H^{k+1}(X^\ast;\pi_k(L))$ such that a cross-section $\sigma\colon X^\ast \to X$ exists if $\omega^1_k = 0$ for all $k$.
\end{thm}

\begin{proof}
By applying Theorem~\ref{T: Obstructions to extension} with $W = Y= X^\ast$, and $A=\emptyset$ we get the result.
\end{proof}

Even  if a section exists on a closed set of the leaf space contained in  the principal part of the foliation, it may happen that it cannot be extended to the whole leaf space (as an example see \cite{Fintushel76} or \cite{Fintushel77}). To solve this new extension problem we need another family of obstructions which we call \emph{second obstructions}.

\begin{thm}\label{T: second obstruction}
Let $(M,\fol)$ be a closed singular Riemannian foliation with $M$ simply con\-nected, and consider the quotient map $\pi\colon M\to M^\ast$. Suppose that the homotopy fiber $F_\pi$ is simple, take $A^\ast\subset M_{\prin}^\ast$ to be a closed subset, such that $(M^\ast,A^\ast)$ is a CW-pair. We also assume we have  defined a cross-section $\sigma\colon A^\ast\to M_{\prin}$. 
Then there is a family of obstructions $\omega^2_k \in H^{k+1}(M^\ast,A^	\ast;\pi_k(F_\pi))$ such that a cross-section $\tilde{\sigma}\colon M^\ast \to M$ extending $\sigma$ exists if $\omega^2_k = 0$ for all $k$.
\end{thm}

\begin{proof}
Since $M$ is simply-connected, then $M^\ast$ is also simply-connected. We apply Theorem~\ref{T: Obstructions to extension} to obtain the desired result.
\end{proof}

In particular, when we cannot distinguish $M^\ast$ from any closed subset of  $M^\ast_{\prin}$ from a homotopical view-point, we get the proof of Corollary~\ref{C: existence of cross sections intro} from Theorem~\ref{T: second obstruction}.

\begin{cor}\label{C: M_prin/F homotopy equiv to M/F and section on M_prin imply section on M/F}
Let $(M,\fol)$ be a closed singular Riemannian foliation on a simply-connected manifold. Suppose that there is a section $\tilde{\sigma}\colon M_{\prin}^\ast\to M_{\prin}$, and the that hypothesis of Theorem~\ref{T: second obstruction} are satisfied. If $M_{\prin}^\ast$ has the same homotopy type as $M^\ast$, then the cross-section $\tilde{\sigma}$ can be extended to a section $\sigma$.
\end{cor}

\begin{remark}
Since the holonomy is only defined for closed leaves (see Section~\ref{SS: holonomy and local projections}), we ask that the foliation is closed in order to ensure the existence of a principal stratum in Theorem~\ref{T: first obstructions}, Theorem~\ref{T: second obstruction}, and Corollary~\ref{C: M_prin/F homotopy equiv to M/F and section on M_prin imply section on M/F}.
\end{remark}

\subsection{Moore-Postnikov towers and Obstruction Theory}\label{S: general Obstruction Theory}
In this section for the sake of completeness we present general statements from Obstruction Theory presented in  \cite[Chapter 4, Obstruction Theory]{Hatcher} used in Section~\ref{S: obstruction theory applied to SRF}. We  being by recalling the definition of a Moore-Postnikov tower.

Recall that a fibration $F\to E\overset{p}{\to} B$ is called a \emph{principal fibration} if there exists a fibration $F'\to E'\overset{p'}{\to} B'$ and a commutative diagram
\[
\begin{tikzcd}
F \ar[r]\ar[d] & E \ar[r]\ar[d] & B\ar[d] &\\
\Omega B' \ar[r] & F' \ar[r] &  E' \ar[r]& B', 
\end{tikzcd}
\]
where all the vertical maps are weak homotopy equivalences.

Given a continuous map $f\colon X\to Y$ between connected spaces a \emph{Moore-Postnikov tower of $f$} is a collection of spaces 
\[
	\cdots\to Z_{n+1}\overset{\alpha_n}{\to}Z_n\to\cdots\to Z_1,
\]
and continuous maps $\alpha_n\colon Z_{n+1}\to Z_n$, $\lambda_n\colon X\to Z_n$ and $\mu_n\colon Z_n\to Y $ such that  
\begin{enumerate}[(i)]
\item $\alpha_n\circ\lambda_{n+1} = \lambda_n$;
\item $\mu_n\circ \alpha_n = \mu_{n+1}$;
\item for $n$ fixed the map $\lambda_n$ induces an isomorphism between $pi_i(X)$ and $\pi_i(Z_n)$ for all  $i<n$, and a surjection for $i = n$;
\item for $n$ fixed the map $\mu_n$ induces an isomorphism between $pi_i(Z_n)$ and $\pi_i(Y)$ for all  $i>n$, and an monomorphism for $i = n$;
\item each map $\alpha_n$ is a fibration with fiber an Eilenberg-MacLane space\linebreak $K(\pi_n(F_f),n)$, where $F_f$ is the homotopy fiber of $f$.
\end{enumerate}
This points are summarized in the following commutative diagram:
\[
\begin{tikzcd}
 & \vdots \arrow{d}{\alpha_3} &\\
 & Z_3 \arrow{d}\arrow{ddr}{\mu_3}& \\
  & Z_2 \arrow{d}\arrow{dr}& \\
X \arrow{r}\arrow{ru}\arrow{ruu}{\lambda_3}  & Z_1 \arrow{r} & Y.
\end{tikzcd}
\]
The idea behind a Moore-Postnikov tower is that we have a series of fibrations between spaces $Z_n$, which at the start they approximate the homotopy type of $Y$ and gradually, as the index $n$ goes to infinity, they approximate the homotopy type of $X$. When the fibrations $\alpha_n\colon Z_{n+1}\to Z_n$ are principal we have say we have a \emph{principal Moore-Postnikov tower for $f$}. In general any map between CW-spaces admits a Moore-Postnikov tower, but the following theorem gives sufficient and necessary conditions for the existence of a principal  Moore-Postnikov tower for a map between CW-spaces:

\begin{thm}[Theorem~4.71 in \cite{Hatcher}, Existence of Moore-Postnikov tower of principal fibrations]\label{T: existence of Moore postnikov towers}
For a given map $f\colon X\to Y$ between connected CW-spaces, a Moore-Postnikov tower of principal fibrations exists if and only if $\pi_1(X)$ acts trivially on $\pi_n(M_f , X)$ for all $ n > 1$, where $M_f$ is the mapping cylinder of $f$.
\end{thm}

With this concepts we can state the following obstruction theorem for extending maps, and include its proof for the sake of completeness.

\begin{thm}[Obstruction Theory in \cite{Hatcher}]\label{T: obstruction for lift for a fibration}
Let $p\colon X\to Y$ be a fibration with fiber $F$, $(W,A)$ a CW-pair with $W$ simply connected. Assume the fibration has a Moore-Postnikov tower of principal fibrations and consider the relative lifting problem:
\begin{center}
\begin{tikzcd}
A \arrow{r}{f} \arrow[hookrightarrow,swap]{d}{i} & X \arrow{d}{p} \\
W \arrow[dashed]{ur}{\tilde{f}} \arrow{r}{g} & Y
\end{tikzcd}
\end{center}
There exists an obstruction $\omega_n\in H^{n+1}(W,A;\pi_n (F) )$, such that a lift\linebreak $\tilde{f}\colon W\to X$ extending $f\colon A\to X$ exists, if $\omega_n=0$ for all $n$.
\end{thm}
\begin{proof}
First we note that since we have a fibration $p\colon X\to Y$, we may take $Z_1$ to be the covering space of $Y$ corresponding to the subgroup $p_\ast (\pi_1 (X))$ of $\pi_1(Y)$. Since $W$ is simply connected we can lift $g$ to $W\to Z_1$, which agrees with $g\circ\lambda_1\colon A\to Z_1$. Since the Moore-Postnikov tower is by principal fibrations, for the inductive step we have a commutative diagram as follows:
\begin{equation}
\begin{tikzcd}
A \ar[r]\arrow[hookrightarrow]{d} & Z_n \ar[r] \ar[d] & PK \ar[d] &[-3.5em]\\
W \ar[r] & Z_{n-1} \ar[r] & K & = K(\pi_n(F), n+1).
\end{tikzcd}
\end{equation}
Here $PK\to K$ is the \emph{path fibration}, defined by fixing a point $b_0\in K$, and letting $PK$  be the space of all curves in $K$ starting at $b_0$, and the letting  $PK\to K$ be the map that sends each path to its end point. Since $Z_n$ is the pullback, the elements in $Z_n$ are pairs consisting of a point in $Z_{n-1}$ and a path from its image in $K$ to the base point in $K$. A lift $W\to Z_n$ therefore amounts to a null-homotopy of the composition $W\to Z_{n-1}\to K$. Since we have already defined such a lift on $A$, we have a null-homotopy of $A\to K$, and the desired null-homotopy of $W\to K$ must extend this null-homotopy on $A$. The map $W\to K$ together with the null-homotopy on $A$ gives a map $W\cup C(A)\to K$, where $C(A)$ is the cone of $A$. Since $K$ is an Eilenberg-MacLane space $K(\pi_n(F), n+1)$, the map $W\cup C(A)\to K$ determines the desired obstruction 
\[
\omega_n\in H^{n+1}(W\cup C(A);\pi_n(F) )= H^{n+1}(W,A;\pi_n(F) ).
\]
If $\omega_n=0$, by construction we have that there is a null-homotopy of $W\to K$ extending the given null-homotopy $A\to K$.

If we succeed in extending the lifts $A\to Z_n$ to lifts $W\to Z_n$ for all $n$, then we obtain a map $W\to \varprojlim Z_n$, to the inverse  limit $\varprojlim Z_n$, extending the given $A\to X\to \varprojlim Z_n$. Let $M$ be the mapping cylinder of  $X\to \varprojlim Z_n$. From the hypothesis that the restriction of  $W\to \varprojlim Z_n\subset M$ to $A$ factors through $X$, this gives a homotopy of  this restriction to the map $A\to X\subset M$. We extend this homotopy to a homotopy of $W\to M$ producing a map $(W,A)\to (M,X)$. Since the map $X\to \varprojlim Z_n$ is a weak homotopy equivalence, then $\pi_i(M,X)=0$ for all $i$, and from the so-called \emph{Compression Lemma} (see Lemma~4.6 in \cite{Hatcher}), we conclude that the map $(W,A)\to (M,X)$ is homotopic relative to $A$ to a map $W\to X$. Hence  the map $W\to X$ extends the given map $A\to X$.
\end{proof}

\begin{thm}(Obstruction to extension)\label{T: Obstructions to extension}
Let $(W,A)$ be a relative $CW$-complex, with $W$ simply connected, and assume we have continuous maps $p\colon X\to Y$, $f\colon A\to X$ and $g\colon W\to Y$ given. Furthermore, suppose that the homotopy fiber $F_p$ of the map $p$  is simple, and that $(M_{p}, X)$ is a simple pair. Then the following are true:
\begin{enumerate}[(i)]
\item\label{T:Family of Obstructions} There is a family of obstructions $\omega_k\in H^{k+1}(W,A;\pi_k(F_p))$ such that there exists a lift $\tilde{f}$ of $f$ such that $p\circ g =\tilde{f}$ and $\tilde{f}|_A = f$, if $\omega_k=0$ for all $k$.
\item\label{T: Unique Obstruction for Eilenberg-McLane space} If $F_p$ is an Eilenberg-McLane space $K(\pi,\ell)$, then there is a unique obstruction $\omega_{\ell}\in H^{\ell+1}(W,A;\pi)$, and the lift $\tilde{f}$ of $f$  exists if and only if $\omega_{\ell}=0$.
\end{enumerate}
\end{thm}

\begin{proof}
We sketch here the proof. For further details we invite the interested reader to see, for example, \cite[Chapter~4]{Hatcher} for a more detailed discussion. First we show that  the pair $(M_p, X)$ is simple only when the pair $(M_{\pi_p},E_{\pi_p})$ is simple. To prove this we start by recalling that  there is a homotopy equivalence $h$ between  $X$ and  the total space $E_{\pi_p}$ of the path space fibration for $p\colon X \to Y$ making the following diagram commute:
\begin{linenomath}
\begin{equation}\label{E: homotopy fiber}
\begin{tikzcd}[column sep=small]
X \arrow{rr}{h} \arrow[swap]{dr}{p} & & E_{\pi_p} \arrow{dl}{\pi_p} \\
 & Y & 
\end{tikzcd}
\end{equation}
\end{linenomath}
For the maps $p\colon X\to Y$ and $\pi_p\colon E_{\pi_p}\to Y$, we have cofibrations $j\colon X\to M_p$ and $i\colon E_{\pi_p}\to M_{\pi_p}$. Furthermore there is also an  homotopy equivalences between the mapping cylinder $M_p$ and  $M_{\pi_p}$ making  the following diagrams commute:
\begin{linenomath}
\begin{equation}\label{E: Mapping cylinders}
\begin{tikzcd}[column sep = small]
	& X\arrow[swap]{dl}{j}\arrow{dr}{p} &	\\
M_p \arrow{rr}{} &	& Y 
\end{tikzcd}
\qquad
\begin{tikzcd}[column sep = small]
	& E_{\pi_p}\arrow[swap]{dl}{\pi_p}\arrow{dr}{i} &	\\
Y &	& M_{\pi_p}\arrow{ll}{}
\end{tikzcd}
\end{equation}
\end{linenomath}

Combining diagram \eqref{E: homotopy fiber} with the diagrams \eqref{E: Mapping cylinders}, we obtain the following diagram which commutes up to homotopy:
\begin{linenomath}
\[
\begin{tikzcd}
X \arrow{r}{} \arrow[hookrightarrow,swap]{d}{j} & E_{\pi_p} \arrow[hookrightarrow]{d}{i} \\
M_p \arrow{r}{} & M_{\pi_p}
\end{tikzcd}
\]
\end{linenomath}
Observe that the arrows going down are cofibrations, and the horizontal ones are homotopy equivalences. Then from the previous commutative diagram and \cite[7.4.2]{Brown}, the homotopy groups $\pi_k(M_p, X)$ are (equivariantly under the action of $\pi_1(X)$) isomorphic to $\pi_k(M_{\pi_q}, E_{\pi_p}) $ (with the action of $\pi_1(E_{\pi_p})= \cong \pi_1(X)$). Thus $(M_p, X)$ is simple if and only if $(M_{\pi_p}, E_{\pi_p})$ is simple. Therefore, for the fibration $\pi_p\colon E_{\pi_p}\to Y$, there exists a Moore-Postnikov tower by principal fibrations, which yield the desired family of obstructions. Thus we may apply Theorem~\ref{T: obstruction for lift for a fibration} to the fibration $\pi_p\colon E_{\pi_p}\to Y$.

We also note that we may apply the last argument in the proof of Theorem~\ref{T: obstruction for lift for a fibration}, and use the fact that the restriction of $W\to \varprojlim Z_n\subset M$ to $A$ factors through $X$, to construct the lift $W\to X$ which extends $A\to X$. Here $M$ is the mapping cylinder of the map $X\to \varprojlim Z_n$.
\end{proof}

\begin{remark}
The reason why in Theorem~\ref{T: Obstructions to extension}~\eqref{T:Family of Obstructions} we have an ``if... then..." statement and on Theorem~\ref{T: Obstructions to extension}~\eqref{T: Unique Obstruction for Eilenberg-McLane space} we have an ``if and only if" statement lies in the fact that for the proofs of these theorems we use a Moore-Postnikov tower of principal fibrations $\cdots\to Z_2\to Z_1\to Y$ for $p$. In the case of Theorem~\ref{T: Obstructions to extension}~\eqref{T:Family of Obstructions} the lifts may be not unique, and in some examples this may yield non trivial $\omega_k$ even when an extension exists. An exception to this, is the case when $F_p$ is an Eilenberg-McLane space (see \cite[Section~4.3]{Hatcher}).
\end{remark}

\begin{remark}
The condition on $W$ being simply connected is used to ensure a unique lift from  $W$ to $Z_1$ in the Moore-Postnikov chain.
\end{remark}

%
%

\bibliographystyle{siam}
\bibliography{Tesis_Doc}

\begin{thebibliography}{10}

\bibitem{Alexandrino}
{\sc M.~M. Alexandrino and R.~G. Bettiol}, {\em Lie groups and geometric
  aspects of isometric actions}, Springer, Cham, 2015.

\bibitem{Alexandrino2012}
{\sc M.~M. Alexandrino, R.~Briquet, and D.~T\"oben}, {\em Progress in the
  theory of singular {R}iemannian foliations}, Differential Geom. Appl., 31
  (2013), pp.~248--267.

\bibitem{Alexandrino2007}
{\sc M.~M. Alexandrino and C.~Gorodski}, {\em Singular {R}iemannian foliations
  with sections, transnormal maps and basic forms}, Ann. Global Anal. Geom., 32
  (2007), pp.~209--223.

\bibitem{Radeschi2016}
{\sc M.~M. Alexandrino and M.~Radeschi}, {\em Mean curvature flow of singular
  {R}iemannian foliations}, J. Geom. Anal., 26 (2016), pp.~2204--2220.

\bibitem{AuslanderKuranishi1957}
{\sc L.~Auslander and M.~Kuranishi}, {\em On the holonomy group of locally
  {E}uclidean spaces}, Ann. of Math. (2), 65 (1957), pp.~411--415.

\bibitem{Bredon}
{\sc G.~E. Bredon}, {\em Introduction to compact transformation groups},
  Academic Press, New York-London, 1972.
\newblock Pure and Applied Mathematics, Vol. 46.

\bibitem{Brown}
{\sc R.~Brown}, {\em Topology and groupoids}, BookSurge, LLC, Charleston, SC,
  2006.

\bibitem{Charlap}
{\sc L.~S. Charlap}, {\em Bieberbach groups and flat manifolds}, Universitext,
  Springer-Verlag, New York, 1986.

\bibitem{Ehresmann}
{\sc C.~Ehresmann}, {\em Les connexions infinit\'esimales dans un espace
  fibr\'e diff\'erentiable}, in Colloque de topologie (espaces fibr\'es),
  {B}ruxelles, 1950, Georges Thone, Li\`ege; Masson et Cie., Paris, 1951,
  pp.~29--55.

\bibitem{Faessler2011}
{\sc D.~Faessler}, {\em The Topology of locally volume collapsed 3-Orbifolds},
  PhD thesis, Ludwig-Maximilians-Universit\"{a}t M\"{u}nchen, June 2011.

\bibitem{FarrellHsiang1983}
{\sc F.~T. Farrell and W.~C. Hsiang}, {\em Topological characterization of flat
  and almost flat {R}iemannian manifolds {$M\sp{n}$} {$(n\not=3,\,4)$}}, Amer.
  J. Math., 105 (1983), pp.~641--672.

\bibitem{Farrell}
{\sc F.~T. Farrell and L.~E. Jones}, {\em Classical aspherical manifolds},
  vol.~75 of CBMS Regional Conference Series in Mathematics, Amer. Math. Soc.,
  Providence, RI, 1990.

\bibitem{FarrellOntaneda}
{\sc F.~T. Farrell, L.~E. Jones, and P.~Ontaneda}, {\em Negative curvature and
  exotic topology}, in Surveys in differential geometry. {V}ol. {XI}, vol.~11
  of Surv. Differ. Geom., Int. Press, Somerville, MA, 2007, pp.~329--347.

\bibitem{FarrellWu2018}
{\sc F.~T. Farrell and X.~Wu}, {\em Riemannian foliation with exotic tori as
  leaves}, Bull. Lond. Math. Soc., 51 (2019), pp.~745--750.

\bibitem{Dirk1981}
{\sc D.~Ferus, H.~Karcher, and H.~F. M\"unzner}, {\em Cliffordalgebren und neue
  isoparametrische {H}yperfl\"achen}, Math. Z., 177 (1981), pp.~479--502.

\bibitem{Fintushel76}
{\sc R.~Fintushel}, {\em Locally smooth circle actions on homotopy
  {$4$}-spheres}, Duke Math. J., 43 (1976), pp.~63--70.

\bibitem{Fintushel77}
\leavevmode\vrule height 2pt depth -1.6pt width 23pt, {\em Circle actions on
  simply connected {$4$}-manifolds}, Trans. Amer. Math. Soc., 230 (1977),
  pp.~147--171.

\bibitem{FloritGoertschesLytchak2015}
{\sc L.~Florit, O.~Goertsches, A.~Lytchak, and D.~T\"oben}, {\em Riemannian
  foliations on contractible manifolds}, M\"unster J. Math., 8 (2015),
  pp.~1--16.

\bibitem{Galaz-GarciaKerin2014}
{\sc F.~Galaz-Garcia and M.~Kerin}, {\em Cohomogeneity-two torus actions on
  non-negatively curved manifolds of low dimension}, Math. Z., 276 (2014),
  pp.~133--152.

\bibitem{Galaz-Garcia2015}
{\sc F.~Galaz-Garcia and M.~Radeschi}, {\em Singular {R}iemannian foliations
  and applications to positive and non-negative curvature}, J. Topol., 8
  (2015), pp.~603--620.

\bibitem{Galaz-Garcia2018}
{\sc F.~Galaz-Garc\'{i}a and M.~Zarei}, {\em Cohomogeneity one topological
  manifolds revisited}, Math. Z., 288 (2018), pp.~829--853.

\bibitem{GeRadeschi2013}
{\sc J.~Ge and M.~Radeschi}, {\em Differentiable classification of 4-manifolds
  with singular {R}iemannian foliations}, Math. Ann., 363 (2015), pp.~525--548.

\bibitem{Geiges92}
{\sc H.~Geiges}, {\em Symplectic structures on {$T^2$}-bundles over {$T^2$}},
  Duke Math. J., 67 (1992), pp.~539--555.

\bibitem{Gromoll}
{\sc D.~Gromoll and G.~Walschap}, {\em Metric foliations and curvature},
  vol.~268 of Progress in Mathematics, Birkh\"auser Verlag, Basel, 2009.

\bibitem{Grove2012}
{\sc K.~Grove and W.~Ziller}, {\em Polar manifolds and actions}, J. Fixed Point
  Theory Appl., 11 (2012), pp.~279--313.

\bibitem{GuijarroWalschap2007}
{\sc L.~Guijarro and G.~Walschap}, {\em When is a {R}iemannian submersion
  homogeneous?}, Geom. Dedicata, 125 (2007), pp.~47--52.

\bibitem{Haefliger84}
{\sc A.~Haefliger}, {\em Groupo\"\i des d'holonomie et classifiants},
  Ast\'erisque,  (1984), pp.~70--97.
\newblock Transversal structure of foliations (Toulouse, 1982).

\bibitem{Hatcher}
{\sc A.~Hatcher}, {\em Algebraic topology}, Cambridge Univ. Press, Cambridge
  [u.a.], 2010.

\bibitem{Hirsch}
{\sc M.~W. Hirsch}, {\em Differential topology}, vol.~33 of Graduate Texts in
  Mathematics, Springer-Verlag, New York, 1994.
\newblock Corrected reprint of the 1976 original.

\bibitem{Hsiang1969}
{\sc W.~C. Hsiang and J.~L. Shaneson}, {\em Fake tori, the annulus conjecture,
  and the conjectures of {K}irby}, Proc. Nat. Acad. Sci. U.S.A., 62 (1969),
  pp.~687--691.

\bibitem{Hsiang}
\leavevmode\vrule height 2pt depth -1.6pt width 23pt, {\em Fake tori}, in
  Topology of {M}anifolds ({P}roc. {I}nst., {U}niv. of {G}eorgia, {A}thens,
  {G}a., 1969), Markham, Chicago, Ill., 1970, pp.~18--51.

\bibitem{Kim1974}
{\sc S.~K. Kim, D.~McGavran, and J.~Pak}, {\em Torus group actions on simply
  connected manifolds}, Pacific J. Math., 53 (1974), pp.~435--444.

\bibitem{KreckLuck2009}
{\sc M.~Kreck and W.~L\"uck}, {\em Topological rigidity for non-aspherical
  manifolds}, Pure Appl. Math. Q., 5 (2009), pp.~873--914.

\bibitem{Lage2018}
{\sc C.~Lange}, {\em Orbifolds from a metric viewpoint}, Geom. Dedicata, 209
  (2020), pp.~43--57.

\bibitem{Lee}
{\sc J.~M. Lee}, {\em Introduction to smooth manifolds}, vol.~218 of Graduate
  Texts in Mathematics, Springer, New York, second~ed., 2013.

\bibitem{Lytchak2010}
{\sc A.~Lytchak}, {\em Geometric resolution of singular {R}iemannian
  foliations}, Geom. Dedicata, 149 (2010), pp.~379--395.

\bibitem{Radeschi15}
{\sc R.~A.~E. Mendes and M.~Radeschi}, {\em A slice theorem for singular
  {R}iemannian foliations, with applications}, Trans. Amer. Math. Soc., 371
  (2019), pp.~4931--4949.

\bibitem{MendesRadeschi2019}
\leavevmode\vrule height 2pt depth -1.6pt width 23pt, {\em A slice theorem for
  singular {R}iemannian foliations, with applications}, Trans. Amer. Math.
  Soc., 371 (2019), pp.~4931--4949.

\bibitem{Moerdijk}
{\sc I.~Moerdijk and J.~Mr\v{c}un}, {\em Introduction to foliations and {L}ie
  groupoids}, vol.~91 of Cambridge Studies in Advanced Mathematics, Cambridge
  University Press, Cambridge, 2003.

\bibitem{Molino}
{\sc P.~Molino}, {\em Riemannian foliations}, vol.~73 of Progress in
  Mathematics, Birkh\"auser Boston, Inc., Boston, MA, 1988.

\bibitem{Morita}
{\sc S.~Morita}, {\em Geometry of differential forms}, vol.~201 of Translations
  of Mathematical Monographs, American Mathematical Society, Providence, RI,
  2001.

\bibitem{Mostert57}
{\sc P.~S. Mostert}, {\em On a compact lie group actiong on a manifold}, Ann.
  of Math. (2), 65 (1957), pp.~447--455.

\bibitem{Oh1982}
{\sc H.~S. Oh}, {\em {$6$}-dimensional manifolds with effective
  {$T\sp{4}$}-actions}, Topology Appl., 13 (1982), pp.~137--154.

\bibitem{Oh1983}
\leavevmode\vrule height 2pt depth -1.6pt width 23pt, {\em Toral actions on
  {$5$}-manifolds}, Trans. Amer. Math. Soc., 278 (1983), pp.~233--252.

\bibitem{Orlik}
{\sc P.~Orlik}, {\em Seifert manifolds}, Lecture Notes in Mathematics, Vol.
  291, Springer-Verlag, Berlin-New York, 1972.

\bibitem{Orlik1970}
{\sc P.~Orlik and F.~Raymond}, {\em Actions of the torus on {$4$}-manifolds.
  {I}}, Trans. Amer. Math. Soc., 152 (1970), pp.~531--559.

\bibitem{Orlik1974}
\leavevmode\vrule height 2pt depth -1.6pt width 23pt, {\em Actions of the torus
  on {$4$}-manifolds. {II}}, Topology, 13 (1974), pp.~89--112.

\bibitem{Radeschi2012}
{\sc M.~Radeschi}, {\em Low diemensional singular {R}iemannian foliations on
  spheres}, PhD thesis, University of Pennsylvania, 2012.

\bibitem{Radeschi2014}
\leavevmode\vrule height 2pt depth -1.6pt width 23pt, {\em Clifford algebras
  and new singular {R}iemannian foliations in spheres}, Geom. Funct. Anal., 24
  (2014), pp.~1660--1682.

\bibitem{Radeschi-notes}
\leavevmode\vrule height 2pt depth -1.6pt width 23pt, {\em {Lecture notes on
  singular {R}iemannian foliations}}, 2017.
\newblock URL:
  \url{https://static1.squarespace.com/static/5994498937c5815907f7eb12/t/5998477717bffc656afd46e0/1503151996268/SRF+Lecture+Notes.pdf}.
  Last visited on 20 April 2018.

\bibitem{Fukuhara83}
{\sc K.~Sakamoto and S.~Fukuhara}, {\em Classification of {$T\sp{2}$}-bundles
  over {$T\sp{2}$}}, Tokyo J. Math., 6 (1983), pp.~311--327.

\bibitem{Scott1983}
{\sc P.~Scott}, {\em The geometries of {$3$}-manifolds}, Bull. London Math.
  Soc., 15 (1983), pp.~401--487.

\bibitem{Siffert2017}
{\sc A.~Siffert}, {\em A new structural approach to isoparametric hypersurfaces
  in spheres}, Ann. Global Anal. Geom., 52 (2017), pp.~425--456.

\bibitem{Tondeur}
{\sc P.~Tondeur}, {\em Geometry of foliations}, vol.~90 of Monographs in
  Mathematics, Birkh\"auser Verlag, Basel, 1997.

\bibitem{Ue90}
{\sc M.~Ue}, {\em Geometric $4$-manifolds in the sense of \uppercase{T}hurston
  and \uppercase{S}eifert $4$-manifolds. \uppercase{I}}, J. Math. Soc. Japan,
  42 (1990), pp.~511--540.

\bibitem{Yeroshkin2014}
{\sc D.~Yeroshkin}, {\em Riemannian {O}rbifolds with {N}on-{N}egative
  {C}urvature}, PhD thesis, University of {P}ennsylvania, 2014.

\bibitem{Zassenhaus1948}
{\sc H.~Zassenhaus}, {\em \"uber einen {A}lgorithmus zur {B}estimmung der
  {R}aumgruppen}, Comment. Math. Helv., 21 (1948), pp.~117--141.

\end{thebibliography}

\end{document}